%% file: main.tex
\documentclass{article}

\usepackage[final]{neurips_2021}

\PassOptionsToPackage{numbers, compress}{natbib}
\usepackage{hyperref}
\hypersetup{colorlinks,
            linkcolor=blue,
            citecolor=blue,
            urlcolor=magenta,
            linktocpage,
            plainpages=false}

\usepackage[utf8]{inputenc} 
\usepackage[T1]{fontenc}    
\usepackage{hyperref}       
\usepackage{url}            
\usepackage{booktabs}       
\usepackage{amsfonts}       
\usepackage{nicefrac}       
\usepackage{microtype}      
\usepackage{xcolor}         

\usepackage{amsmath}
\usepackage{algorithm}
\usepackage{algpseudocode}
\algnewcommand\algorithmicinput{\textbf{Input:}}
\algnewcommand\Input{\item[\algorithmicinput]}
\algdef{SE}[SUBALG]{Indent}{EndIndent}{}{\algorithmicend\ }%
\algtext*{Indent}
\algtext*{EndIndent}
\usepackage{amsthm}
\usepackage{amsfonts}
\usepackage{comment}
\usepackage{graphicx}
\usepackage{hyperref}
\usepackage{color}
\usepackage{wasysym}
\usepackage{bm} 
\usepackage{subcaption}

\newcommand{\tbeta}{\tilde{\beta}}

\newcommand{\tg}{\tilde{g}}
\newcommand{\bg}{\bar{g}}

\newcommand{\rA}{\mathrm{(A)}}
\newcommand{\rB}{\mathrm{(B)}}
\newcommand{\rC}{\mathrm{(C)}}
\newcommand{\rD}{\mathrm{(D)}}
\newcommand{\rE}{\mathrm{(E)}}
\newcommand{\rOne}{\mathrm{(i)}}
\newcommand{\rTwo}{\mathrm{(ii)}}
\newcommand{\rThree}{\mathrm{(iii)}}

\newcommand{\amax}{a_{\max}}

\renewcommand{\le}{~\leq~}

\newcommand{\bx}{\bar{x}}

\renewcommand{\H}{\mathcal{H}}

\newcommand{\supp}{\textbf{support}}
\newcommand{\storm}{$\rm{STORM}$~}
\newcommand{\stormp}{$\rm{STORM}^{+}$~}
\include{defs}

\renewcommand{\bx}{\bar{x}}

\title{$\rm{STORM}^{+}$: Fully Adaptive SGD with Momentum for Nonconvex Optimization}

\author{%
	Kfir Y. Levy \thanks{A Viterbi fellow. Corresponding author} \\
	Technion\\
	\texttt{kfirylevy@technion.ac.il} \\
	\And
	Ali Kavis \\
	EPFL\\
	\texttt{ali.kavis@epfl.ch}
	\And
	Volkan Cevher \\
	EPFL \\
	\texttt{volkan.cevher@epfl.ch} \\
}

\begin{document}

\maketitle

\begin{abstract}
In this work we investigate stochastic non-convex optimization problems where the objective is an expectation over smooth loss functions, and the goal is to find an approximate stationary point. The most popular approach to handling such problems is variance reduction techniques, which are also known to obtain tight convergence rates, matching the lower bounds in this case. Nevertheless, these techniques require a careful maintenance of anchor points in conjunction with appropriately selected ``mega-batchsizes". This leads to a challenging hyperparameter tuning problem, that weakens their practicality. Recently, \citep{cutkosky2019momentum} have shown that one can employ recursive momentum in order to avoid the use of anchor points and large batchsizes, and still obtain the optimal rate for this setting. Yet, their method called $\rm{STORM}$ crucially relies on the knowledge of the smoothness, as well a bound on the gradient norms. In this work we propose $\rm{STORM}^{+}$, a new method that is completely parameter-free, does not require large batch-sizes, and obtains the optimal $O(1/T^{1/3})$ rate for finding an approximate stationary point. Our work builds on the $\rm{STORM}$ algorithm, in conjunction with a novel approach to adaptively set the learning rate and momentum parameters.
\end{abstract}

\section{Introduction}
Over the past decade non-convex models have become  principal tools in ML (Machine Learning), and in data-science. This predominantly includes deep models, as well as Phase Retrieval \citep{candes2015phase}, non-negative matrix factorization \citep{hoyer2004non}, and matrix completion problems \citep{ge2016matrix} among others.

The main workhorse for training ML models is SGD (stochastic gradient descent) and its numerous variants.
One parameter that significantly affects the SGD performance is the learning rate, which often requires a careful and costly hyper-parameter tuning. 
Adaptive approaches to setting the learning rate like AdaGrad \citep{duchi2011adaptive} and Adam \citep{kingma2014adam} as well as non-adaptive heuristics~\citep{j2017on,he2019bag} are very popular in modern ML applications, yet these methods also require some tuning of  hyper-parameters like momentum and the scale of the learning rate schedule.

A popular  SGD heuristic that has proven to be crucial in many applications is the use of \emph{momentum}, i.e., the use of a weighted average of past gradients instead of the current gradient \citep{sutskever2013importance,kingma2014adam}. Although adaptive approaches to setting the momentum have been investigated in the past \citep{srinivasan2018adine,hameed2016back}, principled and theoretically-grounded approaches to doing so are less investigated. Another aspect that has not been extensively studied, which we take into account in this work, is the interplay between learning rate and momentum.

In this work we explore momentum-based adaptive and parameter-free methods for stochastic non-convex optimization problems. Concretely, we focus on the setting where the objective is an \emph{expectation over smooth losses} (see Eq.~\eqref{eq:exp-smooth-loss}), and the goal is to find an approximate stationary point. 

In the general case of smooth non-convex objectives it is known that one can approach a stationary point at a rate of $O(1/T^{1/4})$, where $T$ is the total number of samples \citep{ghadimi2013stochastic}. While this rate is optimal in the general case, it is known that one can obtain an improved rate of $O(1/T^{1/3})$ if the objective is an \emph{expectation over smooth losses}~\citep{fang2018near,zhou2018stochastic,cutkosky2019momentum,tran2019hybrid}. Besides, this rate was recently shown to be tight \citep{arjevani2019lower}.

Nevertheless, most of the methods developed for this setting rely on variance reduction techniques \citep{johnson2013accelerating,zhang2013linear,mahdavi2013mixed,wang2013variance}, which require careful maintenance of anchor points in conjunction with appropriately selected large batchsizes. This leads to a challenging  hyper-parameter tuning problem, weakening their practicality. One exception is the recent \storm algorithm of \cite{cutkosky2019momentum}. 

\storm  does not require large batches nor anchor points; instead, it uses a corrected momentum based gradient update that leads to implicit variance reduction, which in turn facilitates fast convergence. 
Unfortunately, none of the aforementioned methods (including \storm) is parameter-free. Indeed, the knowledge of smoothness parameter together with either the noise variance or a bound on the norm of the gradients are crucial to establish their guarantees. 

In this work, we essentially develop a parameter-free variant of \storm algorithm. We summarize our contributions as follows,
\begin{itemize}
\item We present \stormp, a \emph{parameter-free} momentum based method that ensures the optimal $O(1/T^{1/3})$ rate for the \emph{expectation over smooth losses} setting. Similarly to \storm, our method does not require large-batches nor anchor points.
\item \stormp implicitly adapts to the variance of the gradients. Concretely, it obtains  convergence rate of $O(1/\sqrt{T} + \sigma^{1/3}/T^{1/3})$, which recovers the optimal $O(1/\sqrt{T})$ rate in the noiseless case. 
We also improve over \storm by shaving off a $(\log T)^{3/4}$ factor from the $1/\sqrt T$ term.
\item In \stormp we demonstrate a novel way to set the learning rate by introducing an adaptive interplay between learning rate and momentum parameters. 
\end{itemize}

\section{Related Work}
In the context of stochastic non-convex optimization with general smooth losses, it was shown in \citet{ghadimi2013stochastic} that SGD with an appropriately selected learning rate can obtain a rate of $O(1/T^{1/4})$ for finding an approximate stationary point, which is known to match the respective lower bound \citep{arjevani2019lower}. 
While the method of \citet{ghadimi2013stochastic} requires knowledge of the smoothness and variance parameters, recent works have shown that adaptive methods like AdaGrad are able to obtain this bound in a parameter free manner, as well as to adapt to the variance of the problem \citep{li2019convergence, ward2019adagrad,j2018on}. These results, in a sense, explain the success of adaptive\footnote{An adaptive method is a method that updates its learning rate according to the (noisy) gradient feedback that it receives throughout the training process.} methods like AdaGrad \citep{duchi2011adaptive}, Adam \citep{kingma2014adam}, and RMSProp \citep{tieleman2012lecture} in handling non-convex problems.
 
 The idea of using variance reduction techniques for non-convex problems was first suggested in the context of finite sum problems by \citet{allen2016variance, reddi2016stochastic}, showing a rate of $O(1/T^{1/4})$. This was later improved by \citet{lei2017non} to a rate of  $O(1/T^{3/10})$.
The first works that have obtained the optimal $O(1/T^{1/3})$ for this setting were \citet{fang2018near, zhou2018stochastic}. Additionally,~\citet{fang2018near} shows that the same convergence behavior applies to the more general \emph{expectation over smooth losses} setting (see Eq.~\eqref{eq:exp-smooth-loss}) -- a setting that captures finite-sum problems as a private case.  
 
The \storm algorithm suggested in \citet{cutkosky2019momentum} is the first algorithm to obtain the optimal $O(1/T^{1/3})$ for this setting without the need to maintain anchor points or large batches. Instead, it relies on a clever correction of the momentum by making only one extra call to the oracle, which leads to an implicit variance reduction effect. Moreover, \storm adapts to the variance of the problem by obtaining a rate of $O((\log T)^{3/4}/{\sqrt{T}}+\sigma^{1/3}/T^{1/3})$ without any prior knowledge of variance. However, it needs to know the smoothness parameter and a bound on the gradient norms to set the step size and momentum parameters. Simultaneously to the work of \cite{cutkosky2019momentum}, another paper \citep{tran2019hybrid} have obtained the same optimal bound by proposing a similar update rule.
Note that \cite{tran2019hybrid} does calculate a single anchor point, and it still requires the knowledge of the smoothness and variance parameters.

\section{Setting and Preliminaries} \label{sec:PrelimsAndDefs}
We discuss stochastic non-convex optimization problems where the objective  $f:\reals^d\mapsto \reals$ is of the following form,
$$
f(x) : = \E_{\xi\sim\D}[f(x;\xi)]~,
$$
and $\D$ is an unknown distribution from which we may draw i.i.d.~samples. Our goal is to find an approximate stationary point of $f$, i.e. after $T$ draws from $\D$ we should output a point $\bar{x}\in \reals^d$ such that $\E\|\nabla f(x)\|\leq \text{Poly}(1/T)$. 

We focus on first order methods, i.e., methods that may access the gradients of $f(\cdot,\xi)$, and make the following assumptions regarding the noisy gradients and function values.

\textbf{Bounded values:}~~There exists $B>0$ such that,
\begin{align} \label{eq:bounded-values}
	\max_{x,y\in\reals^d}|f(x)-f(y)|\leq B.
\end{align}
\textbf{Bounded gradients:}~~There exists $G>0$ such that,
\begin{align} \label{eq:bounded-gradient}
	\|\nabla f(x;\xi)\|^2\leq G^2~;\quad \forall x\in\reals^d,\xi\in\supp\{\D\}.
\end{align}
\textbf{Bounded variance:}~~There exists $\sigma>0$ such that,
\begin{align} \label{eq:bounded-variance}
	\E\|\nabla f(x;\xi) - \nabla f(x)\|^2\leq \sigma^2~;\quad \forall x\in\reals^d.
\end{align}
\textbf{Expectation over smooth losses:} There exists $L>0$ such that,
\begin{align} \label{eq:exp-smooth-loss}
\|\nabla f(x;\xi)-\nabla f(y;\xi)\|\leq L\|x-y\|~;\quad \forall x,y\in\reals^d,\xi\in\supp\{\D\}~.
\end{align}
The last assumption also implies that the expected loss $f(\cdot)$ is $L$ smooth. A property of smooth functions that we will exploit throughout the paper is the following,
\begin{align} \label{eq:smoothness}
f(y)\leq f(x) + \nabla f(x)^\top(y-x)+ (L/2)\|y-x\|^2~; \quad \forall x,y\in\reals^d
\end{align}

In the rest of this manuscript, $\nabla f(x;\xi)$ relates to gradients with respect to $x$, i.e., $\nabla: = \nabla_x$. We use $\|\cdot\|$ to denote the Euclidean norm, and  $x^*$ denotes a global minima of $f(\cdot)$, i.e., $x^* = \min_{x \in \mathbb R^d} f(x)$.

\section{Method}
In this section we present \stormp (STochastic Recursive Momentum $+$):  a parameter-free stochastic optimization method that finds approximate stationary points at an optimal rate. 
We describe our method in Alg.~\ref{alg:stormp} and Eq.~\eqref{eq:stormP_Choices}, and state its guarantees in Theorem~\ref{thm:StormpMain}.

\paragraph{The original \storm algorithm:}
The original \storm template of \cite{cutkosky2019momentum} relies on an SGD-style update with a corrected momentum.
Concretely, the idea is to maintain a gradient estimate $d_t$ which is a \emph{corrected} weighted average of past stochastic gradients, and then update the iterates similarly to SGD, 
\begin{align}\label{eq:stormMain}
x_{t+1} = x_t-\eta_t d_t~.
\end{align}
Standard momentum is a weighted average of past gradients,
$$
d_{t} = a_t\nabla f(x_{t},\xi_{t}) + (1-a_t)d_{t-1}~;~~\text{where~} a_t\in[0,1]~.
$$
Under this construction, $d_t$ is generally a biased estimate of $\nabla f(x_t)$. In \storm it is suggested to add a correction term ,$(1-a_t)(\nabla f(x_{t},\xi_{t}) - \nabla f(x_{t-1},\xi_{t}))$, which leads to the following update rule (again, $a_t\in[0,1]$),
\begin{align}\label{eq:stormMomentum}
d_{t} = \nabla f(x_{t},\xi_{t}) + (1-a_t)(d_{t-1} - \nabla f(x_{t-1},\xi_{t}))~, \tag{Corrected Momentum}
\end{align}
The correction term plays a crucial role here: it exploits the smoothness of $f(\cdot,\xi)$ in a way that leads to a variance reduction effect. To see this effect one can inspect  the error of the momentum $d_t$ compared to the exact gradient at $x_t$,
$$
\eps_t: = d_t-\nabla f(x_t)~.
$$
The \storm update rule induces the following error dynamics,
$$
\eps_{t} = (1-a_t)\eps_{t-1} + a_t(\nabla f(x_t,\xi_t) - \nabla f(x_t)) + (1-a_t)Z_t
$$
where $Z_t: =(\nabla f(x_t,\xi_t) - \nabla f(x_{t-1},\xi_t))-(\nabla f(x_t) - \nabla f(x_{t-1}))$. 
Due to the smoothness of the objective we have $\|Z_t\|\leq O(\|x_t-x_{t-1}\|) = O(\eta_{t-1}\|d_{t-1}\|)$. Intuitively, as we approach a stationary point (and use a small enough learning rate) then $\eta_{t-1}\|d_{t-1}\|$ decreases which in turn reduces the magnitude of $Z_t$'s. Moreover, the second term in the above dynamics, $a_t(\nabla f(x_t,\xi_t) - \nabla f(x_t))$,
can be controlled by choosing a small enough momentum $a_t$.  Thus, carefully controlling the learning rate and momentum parameters leads to a variance reduction effect which facilitates fast convergence.

The original \storm paper \citep{cutkosky2019momentum} makes the following choices,
\begin{align}\label{eq:stormChoices}
\eta_t = \theta/{\left(w+\sum_{i=1}^t \|g_i\|^2\right)^{1/3}}~~~~~~~~~\text{\&~}~~~~~~~~a_t =cL^2\eta_{t-1}^2~, 
\end{align}
where we denote $g_t: =\nabla f(x_t,\xi_t)$. The above choice of learning rate is inspired by AdaGrad \cite{duchi2011adaptive}, which also sets the learning rate inversely proportional to the cumulative square norms of past gradients.
Note that $\theta$ and $w$ are constants that depend on the smoothness of the objective $L$, as well as on the bound on the gradients $G$, and $c$ is an absolute constant independent of the problem's characteristics. These choices of the constants and especially the choice of $a_t\propto L^2 \eta_{t-1}^2$ is crucial for the analysis of the original \storm. In fact, the convergence proof for \storm breaks down unless we encode this prior knowledge  into $\eta_t$ and $a_t$. Next, we describe our parameter-free version.

\paragraph{Our \stormp algorithm:}
\stormp relies on the original \storm template described in Equations~\eqref{eq:stormMain} and \eqref{eq:stormMomentum}, with the following parameter-free choices of learning rate and momentum parameter,
\begin{align}\label{eq:stormP_Choices}
\eta_t = 1/\left(\sum_{i=1}^t \|d_i\|^2/a_{i+1}\right)^{1/3}~~~~~~~~\text{\&~}~~~~~~~~
a_t = 1/\left(1+\sum_{i=1}^{t-1} \|g_i\|^2\right)^{2/3}~,
\end{align}
where again we denote $g_t: =\nabla f(x_t,\xi_t)$. Note that in contrast to the original \storm our adaptive  learning rate
builds on history of estimates $\{d_1,\ldots,d_t\}$ as well as on the momentum parameters $\{a_1,\ldots,a_{t+1}\}$. Our momentum term is similar to the adaptive choice of \storm, yet it does not require a bound on the gradients nor on the smoothness parameter, which was crucial for the original analysis.
Finally, note that the above choice ensures $a_t\in[0,1]$.

\begin{algorithm}[t]
\caption{\stormp} \label{alg:stormp}
\begin{algorithmic}[1]
\Input{ \#iterations $T$ , $x_1 \in \reals^d$}
\State Sample $\xi_1$ and set $d_1=g_1 = \nabla f(x_1,\xi_1)$
\For{$t = 1, ..., T$}
	\Indent
		\State $a_{t+1} \gets  1/\left(1+\sum_{i=1}^{t} \|g_i\|^2\right)^{2/3}~~~~\&~~~~\eta_t \gets 1/\left(\sum_{i=1}^t \|d_i\|^2/a_{i+1}\right)^{1/3}$
		\State $x_{t+1} \gets x_t - \eta_t d_t$
                 \State Sample $\xi_{t+1}$ and set $g_{t+1}: = \nabla f(x_{t+1};\xi_{t+1})$, and $\tg_{t}: = \nabla f(x_t;\xi_{t+1})$
		\State $d_{t+1} \gets g_{t+1} + (1-a_{t+1})(d_t - \tg_t)$
	\EndIndent
\EndFor 
\State Choose $\bar{x}_T$ uniformly at random from $\{x_1,\ldots,x_T\}$\\
\Return  $\bx_T$
\end{algorithmic}
\end{algorithm}

For completeness we present our method in Alg.~\ref{alg:stormp}, where it can be seen that \stormp is a combination of the original \storm template (Equations~\eqref{eq:stormMain} and \eqref{eq:stormMomentum}) together with the specific choices of $\eta_t$ and $a_t$ appearing in Eq.~\eqref{eq:stormP_Choices}.
Note that the solution that \stormp outputs is a point chosen uniformly at random among all iterates, which is quite standard in (stochastic) non-convex optimization.

\emph{Notation:} In Alg.~\ref{alg:stormp} and throughout the rest of the paper we will employ the following notation,
$$
g_t : = \nabla f(x_t,\xi_t)~;~~\tg_t: = \nabla f(x_t,\xi_{t+1})~;~~\bg_t: =  \nabla f(x_t)~.
$$
Now, we are at a position to present our main theorem regarding \stormp (Alg.~\ref{alg:stormp}):
\begin{Thm}\label{thm:StormpMain}
Under the assumption in Eq.~\eqref{eq:bounded-values},~\eqref{eq:bounded-gradient},~\eqref{eq:bounded-variance} and~\eqref{eq:exp-smooth-loss} in Section~\ref{sec:PrelimsAndDefs}, \stormp ensures,
$$
\E\|\nabla f(\bx_T)\| \leq O\left(\frac{M}{\sqrt{T}} + \frac{\kappa\sigma^{1/3}}{T^{1/3}} \right)~, 
$$
where $\kappa =O( B^{3/4}+L^{3/2})$;~ $M= O(1+L^{9/4}+B^{9/8}+G^{5} + (LG^4)^{3/2})$, and
the expectation is with respect to the randomization of the samples as well as the algorithm's.
\end{Thm}
Theorem~\ref{thm:StormpMain} demonstrates that in the stochastic case \stormp achieves the optimal $O(1/T^{1/3})$ rate for our setting.
Moreover, it can be seen that \stormp implicitly adapts to the variance of the noise; in the noiseless case where $\sigma=0$, \stormp recovers the optimal $O(1/\sqrt{T})$ rate. We note that scaling the learning rate by some (absolute) constant factor may enable us to obtain better dependence on $L$ and $B$.

\section{Analysis}
In this section we  provide the convergence analysis of the \stormp algorithm.
We begin with the analysis in the offline case where $\sigma=0$, and establish a convergence rate of $O(1/\sqrt{T})$ in Section~\ref{sec:OfflineAnalysis} for completeness.
In Section~\ref{sec:StochasticAnalysisSimple}, we introduce a simplified version of \stormp, with a non-adaptive momentum parameter of the form $a_{t+1}: = 1/t^{2/3}$. Due to simplicity and space limitations, it is inconvenient to share the full proof of \stormp, and this simplified version enables us to illustrate the main steps of the original proof.
We show that this version achieves a convergence rate of $O(1/T^{1/3})$ in the stochastic case (though it does not adapt to the variance).
Finally, in Section~\ref{sec:OfflineAnalysis} we provide a proof sketch for \stormp in Alg.~\ref{alg:stormp} that establishes the result in Theorem~\ref{thm:StormpMain}. 
\subsection{Offline Case}~\label{sec:OfflineAnalysis}
Here we analyze \stormp in the case where $\sigma=0$, and demonstrate a rate of $O(1/\sqrt{T})$ for finding an approximate stationary point.
\begin{Thm}
	Let $f$ satisfy Eq.~\eqref{eq:bounded-values},~\eqref{eq:exp-smooth-loss} and $\bar x_T$ be generated after running Alg.~\ref{alg:stormp} for $T$ iterations under deterministic oracle. Then it holds that,
	\begin{align*}
		\E\|\nabla f(\bx_T)\| \leq O(\sqrt{1+L^3+B^{9/4}}/\sqrt{T})~.
	\end{align*}
	where we take expectation due to randomness governing the generation of $\bar x_T$ (see line 8 in Alg.~\ref{alg:stormp}).
\end{Thm}
\begin{proof}
In the case where $\sigma=0$ one can  directly show by induction that $d_t=\bg_t: = \nabla f(x_t)$.
So the update rule becomes $x_{t+1} = x_t - \eta_t \bg_t$.
Now, using the smoothness of the objective implies,  
\begin{align*}
\Delta_{t+1}-\Delta_t = 
f(x_{t+1})-f(x_t) 
&\leq
 -\eta_t \|\bg_t\|^2 + {L\eta_t^2} \|\bg_t\|^2/2~,
\end{align*}
here we denoted $\Delta_t: = f(x_t)-f(x^*)$, where $x^*\in\argmin f(x)$.
Dividing by $\eta_t$, re-arranging and summing gives,
\begin{align}\label{eq:OfflineCase}
\sum_{t=1}^T\|\bg_t\|^2
& \leq 
\frac{\Delta_1}{\eta_1} -\frac{\Delta_{T+1}}{\eta_T} +
\sum_{t=2}^T\left(\frac{1}{\eta_t}-\frac{1}{\eta_{t-1}}\right)\Delta_t
+
\frac{L}{2}\sum_{t=1}^T\eta_t \|\bg_t\|^2 \nonumber\\
&\leq
\frac{B}{\eta_1} + B\sum_{t=2}^T\left(\frac{1}{\eta_t}-\frac{1}{\eta_{t-1}}\right)
+
\frac{L}{2}\sum_{t=1}^T\frac{\|\bg_t\|^2}{\left(\sum_{i=1}^t\|\bg_i\|^2\right)^{1/3}}  \nonumber\\
&\leq 
\frac{B}{\eta_T} +L\left(\sum_{i=1}^t\|\bg_i\|^2\right)^{2/3} 
\leq
B\left(\sum_{t=1}^T \|\bg_t\|^2/a_{t+1}\right)^{1/3} +L\left(\sum_{t=1}^T\|\bg_t\|^2\right)^{2/3}  \nonumber\\
&\leq
B\left(1+\sum_{t=1}^T \|\bg_t\|^2\right)^{2/9}\left(\sum_{t=1}^T \|\bg_t\|^2\right)^{1/3} +L\left(\sum_{t=1}^T\|\bg_t\|^2\right)^{2/3} 
\end{align}
where the second inequality uses  $\eta_t =\left(\sum_{i=1}^t\|\bg_i\|^2/a_{i+1}\right)^{-1/3}  \leq \left(\sum_{i=1}^t\|\bg_i\|^2\right)^{-1/3}$ which holds since $d_t=\bg_t$ and $a_t\leq 1$. 
We also use that $\Delta_t \in [0,B]$ together with $\eta_{t}^{-1} - \eta_{t-1}^{-1}\geq 0$. 
The third inequality uses Lemma~\ref{lem:technical-1} below; and the last inequality uses $1/a_{t+1} \leq (1/a_{T+1})=  \left(1+\sum_{t=1}^T \|\bg_t\|^2\right)^{2/3}$, which holds since $a_t$ is monotonically non-increasing.

By treating the inequality in Eq.~\eqref{eq:OfflineCase} as a polynomial of $x = \sum_{t=1}^T\|\bg_t\|^2$, one can derive the following bound,
$
\sum_{t=1}^T\|\bg_t\|^2 \leq O(1+L^3+B^{9/4})~.
$
Using the definition of $\bx_T$ as well as Jensen's inequality implies,
$$
\E\|\nabla f(\bx_T)\|:=\E\|\bg(\bx_T)\| \leq \sqrt{\E\|\bg(\bx_T)\|^2 } =  \sqrt{\sum_{t=1}^T\|\bg_t\|^2/T } \leq O(\sqrt{1+L^3+B^{9/4}}/\sqrt{T})~.
$$
which establishes the bound. In the proof we have used the technical lemma below,
\begin{Lem} \label{lem:technical-1}
	Let $b_1 > 0$, $b_2, ..., b_n\geq 0 $ be a sequence of real numbers, $p \in (0,1)$ be a real number.
	\begin{align*}
		 \sum_{i=1}^{n} \frac{b_i}{ \left( \sum_{j=1}^{i} b_j \right)^{p} } \leq \frac{1}{1-p} \left( \sum_{i=1}^{n} b_i \right)^{1-p}
	\end{align*}
\end{Lem}
\end{proof}
\subsection{Stochastic Case Analysis of Simplified \stormp} \label{sec:StochasticAnalysisSimple}
Here we analyze a simplified version of \stormp in the stochatic setting. 
While this version does not adapt to the noise variance, it exhibits the optimal rate of
$O(1/T^{1/3})$ in the stochastic case, and its analysis illustrates some of the main ideas that we employ in the proof of the fully adaptive \stormp (which is more involved).

The  version that we analyze here differs from \stormp in the choice of the momentum parameters. 
Here we choose $a_1=1$ and $a_{t+1}=1/t^{2/3}~;\forall t\geq 1$, in contrast to the adaptive choice that we make in Alg.~\ref{alg:stormp}. 
Note that we keep the same expression for the step size, $\eta_t =  1/\left(\sum_{i=1}^t \|d_i\|^2/a_{i+1}\right)^{1/3}$.
\begin{Thm}
	Under Assumptions in Eq.~\eqref{eq:bounded-values},~\eqref{eq:bounded-gradient},~\eqref{eq:bounded-variance} and~\eqref{eq:exp-smooth-loss}, \emph{simplified} \stormp ensures,
	\begin{align*}
		\E\|\nabla f(\bx_T)\| = O( \sqrt{L^3 +\sigma^2+ B^{3/2}}/T^{1/3})~,
	\end{align*}
\end{Thm}

\begin{proof}
The proof is composed of two parts. In the first we bound the cumulative expectation of errors $\E\sum_{t=1}^T\|\eps_t\|^2$, where $\eps_t$ is the difference between the corrected momentum $d_t$ and the exact gradient $\bg_t$, i.e.~$\eps_t=d_t-\bg_t$. Thus, in the first part we relate the above sum to the sum of exact gradients $\E\sum_{t=1}^T\|\bg_t\|^2$.
Then, in the second part we divide into two sub-cases the first where $\E\sum_{t=1}^T\|\eps_t\|^2\leq (1/2)\E\sum_{t=1}^T\|\bg_t\|^2$ and its complement. 
In one of these sub-cases we also use the smoothness of the objective together with the update rule, similarly to what we do in Eq.~\eqref{eq:OfflineCase}.

\paragraph{First Part: Bounding $\E\sum_{t=1}^T\|\eps_t\|^2$.} The  update rule for $d_t$ induces the following error dynamics,
\begin{align}\label{eq:ErrorDynamics}
\eps_{t} = (1-a_t)\eps_{t-1} + a_t(g_t - \bg_t) + (1-a_t)Z_t
\end{align}
where $Z_t: = (g_t-\tg_{t-1}) - (\bg_t-\bg_{t-1})$.  Letting $\H_t$ be the history to time $t$, i.e.,
$\H_t: = \{x_1,\xi_1,\xi_2,\xi_3\ldots, \xi_t\}$ and recalling that both $a_t$ and $x_t$ depend on history up to $t-1$, i.e., $\mathcal H_{t-1}$,
immediately implies that  $\E[a_t(g_t - \bg_t)\vert \H_{t-1}] =\E[(1-a_t)Z_t\vert \H_{t-1}] =0$, as well as
$\E[(1-a_t)\eps_{t-1}\vert \H_{t-1}] = (1-a_t)\eps_{t-1}$.
 
Thus, 
taking the square of the above equation and then taking the expectation gives,
\begin{align}\label{eq:EpsRecurse1}
	\E \norm{\eps_t}^2 
	&\leq (1-a_t)^2 \E \norm{\eps_{t-1}}^2 +\|(1-a_t)Z_t + a_t(g_t-\bg_t) \|^2 \nonumber\\
	&\leq
	(1-a_t)^2 \E \norm{\eps_{t-1}}^2+2(1-a_t)^2\norm{Z_t}^2 + 2a_t^2\E\|g_t-\bg_t\|^2 \nonumber\\
	&\leq
	(1-a_t) \E \norm{\eps_{t-1}}^2+ 8L^2\E \eta_{t-1}^2 \norm{d_{t-1}}^2 + 2a_t^2\sigma^2~,
\end{align}
where the second line uses $\|b+c\|^2\leq 2\|b\|^2+2\|c\|^2$, and the last line uses $\E\|g_t-\bg_t\|^2\leq \sigma^2$ and $(1-a_t)\in[0,1]$, as well as
the smoothness assumption that implies $\|Z_t\| \leq \|g_t-\tg_{t-1}\| + \|\bg_t-\bg_{t-1}\|\leq 2L\|x_t-x_{t-1}\| = 2L\eta_{t-1}\|d_{t-1}\|$.

Dividing Eq.~\eqref{eq:EpsRecurse1} by $a_t$ and re-arranging implies,
\begin{align*}
	\E \norm{\eps_{t-1}}^2 
	&\leq 
	\frac{1}{a_t} (\E \norm{\eps_{t-1}}^2 - \E \norm{\eps_{t}}^2)
	 +8L^2\E [\eta_{t-1}^2 \norm{d_{t-1}}^2/a_t] + 2a_t\sigma^2~.
\end{align*}
Summing the above, and using $\eps_0:=0$ gives,
\begin{align}\label{eq:EpsRecurse2}
	\E \sum_{t=1}^T \norm{\eps_{t-1}}^2 
	&\leq 
	\underbrace{-\frac{\E \norm{\eps_{T}}^2}{a_T}}_{\rA}
	+
	\underbrace{\sum_{t=1}^{T-1} (\frac{1}{a_{t+1}}-\frac{1}{a_{t}}) \E \norm{\eps_{t}}^2}_{\rB}
	 +
	 8L^2\underbrace{\E [\sum_{t=1}^T \eta_{t-1}^2 \norm{d_{t-1}}^2/a_t]}_{\rC}
	 + 
	 2\sigma^2\underbrace{\sum_{t=1}^T a_t }_{\rD}
\end{align}
Next we bound all the term on the RHS of the above equation:\\

\textbf{Bounding $\rA$:} Since $a_T\leq 1$ we can bound $-{\E \norm{\eps_{T}}^2}/{a_T}\leq -\E \norm{\eps_{T}}^2$ \\

\textbf{Bounding $\rB$:} Note that $G(z) = z^{2/3}$ is a concave  function in $\reals_{+}$. Thus applying the gradient inequality  implies that $\forall z_1,z_2\geq 0$  we have $(z_1+z_2)^{2/3} - z_1^{2/3} \leq \frac{2}{3} z_1^{-1/3}z_2$. Hence, for all $ t\geq 2$,
$$
	{1}/{a_{t+1}}-{1}/{a_{t}} = t^{2/3} - (t-1)^{2/3} \leq {2(t-1)^{-1/3}}/{3} \leq {2}/{3}~.
$$
Moreover, $1/a_{2} - 1/a_1 = 0$.
These imply that $\rB \leq (2/3)\E \sum_{t=1}^T \norm{\eps_{t}}^2$.

\textbf{Bounding $\rC$:} By definition of $\eta_{t}$ we have,
\begin{align*}
\rC = \E\sum_{t=1}^T\frac{\norm{d_{t-1}}^2/a_t}{\left(\sum_{i=1}^{t-1} \|d_i\|^2/a_{i+1}\right)^{2/3}} \leq 3\E\left(\sum_{t=1}^{T-1} \|d_t\|^2/a_{t+1}\right)^{1/3}\leq 3T^{2/9} \left(\E\sum_{t=1}^{T} \|d_t\|^2\right)^{1/3}~.
\end{align*}
where the first inequality uses Lemma~\ref{lem:technical-1}, and the second inequality uses $1/a_t\leq 1/a_{T
+1} \leq T^{2/3}$ as well as Jensen's inequality with respect to the concave function $U(z) = z^{1/3}$, defined over $\reals_{+}$.

\textbf{Bounding $\rD$:} Lemma~\ref{lem:technical-1} immediately implies that $\rD =1+\sum_{t=1}^{T-1} 1/t^{2/3}\leq 1+3T^{1/3}\leq 4T^{1/3}$.

Plugging these bounds into Eq.~\eqref{eq:EpsRecurse2} and re-arranging yields,
\begin{align}\label{eq:EpsFinalPartA}
	\E \sum_{t=1}^T \norm{\eps_{t}}^2 
	&\leq 
	72L^2T^{2/9} \left(\E\sum_{t=1}^{T} \|d_t\|^2\right)^{1/3}
	+
	24\sigma^2 T^{1/3}~.
\end{align}

\paragraph{Second Part: Bounding $\E\sum_{t=1}^T\|\bg_t\|^2$.}
Here we use the bound of Eq.~\eqref{eq:EpsFinalPartA} in order to bound the sum of square gradients.
Let us divide into two sub-cases.  \\

\textbf{Case 1 : Assume that $\E\sum_{t=1}^T \norm{\eps_t}^2 \geq \frac{1}{2}\E\sum_{t=1}^T\|\bg_t\|^2 $.} Combining the condition of Case 1 with $\|d_t\|^2 \leq 2\|\bg_t\|^2 + 2\|\eps_t\|^2$ (due to $d_t = \bg_t+\eps_t$), implies that $\E\sum_{t=1}^T\|d_t\|^2 \leq 6\E\sum_{t=1}^T\norm{\eps_t}^2$. Plugging this inside Eq.~\eqref{eq:EpsFinalPartA} yields,
\begin{align*}
	\E \sum_{t=1}^T \norm{\eps_{t}}^2 
	&\leq 
	72L^2T^{2/9} \left(6\E\sum_{t=1}^{T} \norm{\eps_{t}}^2 \right)^{1/3}
	+
	24\sigma^2 T^{1/3}~.
\end{align*}
The above immediately implies that $\E \sum_{t=1}^T \norm{\eps_{t}}^2 \leq O((L^3 + \sigma^2)T^{1/3})$, and due to the condition of Case 1 we therefore have,
$\E \sum_{t=1}^T \norm{\bg_{t}}^2 \leq O((L^3 + \sigma^2)T^{1/3})$. This concludes the first case.

\textbf{Case 2 : Assume that $\E\sum_{t=1}^T \norm{\eps_t}^2 \leq \frac{1}{2}\E\sum_{t=1}^T\|\bg_t\|^2 $.}
Combining the condition of Case 2 with $\|d_t\|^2 \leq 2\|\bg_t\|^2 + 2\|\eps_t\|^2$ (due to $d_t = \bg_t+\eps_t$), implies that $\E\sum_{t=1}^T\|d_t\|^2 \leq 3\E\sum_{t=1}^T\norm{\bg_t}^2$. 

Now using the update rule $x_{t+1}=x_t-\eta_t d_t$ together with smoothness of $f(\cdot)$, one can show in a similar manner to our derivation of Eq.~\eqref{eq:OfflineCase} the following bound (we defer this to the appendix),
\begin{align}\label{eq:SumOfGradsStochastic}
	\sum_{t=1}^{T} \|\bg_t\|^2 
	 &\leq
	 \sum_{t=1}^{T} \|\eps_t\|^2 + 2BT^{2/9}\left(\sum_{t=1}^{T} \|d_t\|^2\right)^{1/3} + \frac{3}{2}L\left(\sum_{t=1}^T \|d_t\|^2\right)^{2/3} 
\end{align}
Taking the expectation of the above equation and plugging in  $\E\sum_{t=1}^T\|d_t\|^2 \leq 3\E\sum_{t=1}^T\norm{\bg_t}^2$ as well as $\E\sum_{t=1}^T \norm{\eps_t}^2 \leq \frac{1}{2}\E\sum_{t=1}^T\|\bg_t\|^2 $ gives,
\begin{align}\label{eq:SumOfGradsStochasticFinal}
	\E\sum_{t=1}^{T} \|\bg_t\|^2 
	 &\leq
	 \frac{1}{2}\E\sum_{t=1}^{T} \|\bg_t\|^2 + 2BT^{2/9}\left(3\E\sum_{t=1}^{T} \|\bg_t\|^2\right)^{1/3} + \frac{3}{2}L\left(3\E\sum_{t=1}^T \|\bg_t\|^2\right)^{2/3} 
\end{align}
where we also used Jensen's inequality with respect to he concave functions $z^{1/3}$ and $z^{2/3}$ defined over $\reals_{+}$.
The above immediately implies, $\E \sum_{t=1}^T \norm{\bg_{t}}^2 \leq O(L^3 + B^{3/2}T^{1/3})$. This concludes the second case.

\paragraph{Summary.} We have shown that $\E \sum_{t=1}^T \norm{\bg_{t}}^2 \leq O((L^3 +\sigma^2+ B^{3/2})T^{1/3})$, combining this with the definition of $\bx_T$ and using Jensen's inequality similarly to what we did in the offline analysis provides,
$$
\E\|\nabla f(\bx_T)\| = O( \sqrt{L^3 +\sigma^2+ B^{3/2}}/T^{1/3})~,
$$
which concludes the proof.
\end{proof}
\subsection{Stochastic Case Analysis of \stormp} \label{sec:StochasticAnalysis}
Finally, we provide a sketch of the proof for the \stormp algorithm appearing in Alg.~\ref{alg:stormp}.
In a high level, the analysis follows similar lines to that of of simplified \stormp's appearing in Section~\ref{sec:StochasticAnalysisSimple}. 

There are two extra challenges compared to the analysis of simplified \stormp:
\begin{enumerate}
	\item Now $a_t$ is a random variable that depends on the noisy samples.  
	\item The differences $1/a_{t+1} -1/a_t$ are not necessarily smaller than $1$. 
\end{enumerate}
Recall that in the analysis appearing in Section~\ref{sec:StochasticAnalysisSimple} we used $1/a_{t+1} -1/a_t\leq 2/3$,
which was crucial to bounding term  $\rB$.

Among the tools that we use to address the first challenge is a version of Young's inequality, that we mention in the appendix. 
To cope with the second challenge, when we bound the expectation of $\sum_{t=1}^T\|\eps_t\|^2$, it is split into two,
$$
\sum_{t=1}^T\|\eps_t\|^2 = \sum_{t=1}^{\tau^*}\|\eps_t\|^2  +  \sum_{t=\tau^*+1}^{T}\|\eps_t\|^2
$$
where $\tau^*$ is a time-step after which we can ensure that $1/a_{t+1} -1/a_t\leq 2/3$. Next we proceed with the proof sketch.

\begin{proof}[Proof Sketch of Theorem~\ref{thm:StormpMain}]
The proof is composed of three parts: \textbf{(a)} In the first part we bound the cumulative expectation of errors $\E\sum_{t=1}^{\tau^*}\|\eps_t\|^2$, where $\eps_t: = d_t-\bg_t$, and $\tau^*$ is a stopping time after which we can ensure that $1/a_{t+1} -1/a_t\leq 2/3$. \textbf{(b)} In the second part we use our bound on $\E\sum_{t=1}^{\tau^*}\|\eps_t\|^2$ in order to bound the total sum of square errors, $\E\sum_{t=1}^{T}\|\eps_t\|^2$.
\textbf{(c)}  Then, in the last part we divide into two sub-cases the first where $\E\sum_{t=1}^T\|\eps_t\|^2\leq (1/2)\E\sum_{t=1}^T\|\bg_t\|^2$ and its complement. 
In one of these sub-cases we also use the smoothness of the objective together with the update rule, similarly to what we do in Eq.~\eqref{eq:OfflineCase}.

\paragraph{First Part: Bounding $\E\sum_{t=1}^{\tau^*}\|\eps_t\|^2$.}
Recall the error dynamics of \stormp appearing in Eq.~\eqref{eq:ErrorDynamics}. Taking the square and summing up to some $\tau^*\in[T]$ enables to bound,
\begin{align*}
\sum_{t=1}^{\tau^*}\|\eps_t\|^2
 \leq \sum_{t=1}^{\tau^*}(1-a_t)\|\eps_{t-1}\|^2 +2\sum_{t=1}^{\tau^*}\|Z_t\|^2 + 2\sum_{t=1}^{\tau^*}a_t^2 \|g_t-\bg_t\|^2
+ \sum_{t=1}^{\tau^*}M_t~,
\end{align*}
where $M_t = 2 \langle (1 - a_t) \epsilon_{t-1}, a_t(g_t - \bg_t) + (1 - a_t)Z_t \rangle$ is a martingale difference sequence such that $\E[M_t\vert \H_{t-1}]=0$, where $\H_t$ is the history to time $t$, i.e., $\H_t: = \{x_1,\xi_1,\xi_2,\xi_3\ldots, \xi_t\}$. Also, recall that $Z_t: = (g_t-\tg_{t-1}) - (\bg_t-\bg_{t-1})$.

Now let us  define $\beta := \min\{1,1/G^4\}$, and  $\tau^* = \max\{t\in[T]: a_t\geq \beta\}$. Recalling that $a_{t+1}$ is measurable with respect to $\H_t$ implies that $\tau^*\in[T]$ is a stopping time.

Re-arranging the above and using the definition of $\tau^*$ implies,
\begin{align*}
\beta\sum_{t=1}^{\tau^*}\|\eps_t\|^2 \leq \|\eps_{\tau^*}\|^2 + \sum_{t=1}^{\tau^*-1}a_{t+1}\|\eps_t\|^2 \leq 2\underbrace{\sum_{t=1}^{T}\|Z_t\|^2}_{\rOne} + 2\underbrace{\sum_{t=1}^{T}a_t^2 \|g_t-\bg_t\|^2}_{\rTwo} + 
\underbrace{\sum_{t=1}^{\tau^*}M_t}_{\rThree}
\end{align*}
where we used $\tau^*\leq T$, as well as $\beta\leq 1$.
Next we bound the expected value  of the above terms.\\

\textbf{Bounding $\rOne$.} As in the previous section, the smoothness property implies that $\|Z_t\|^2\leq 4L^2\eta_{t-1}^2\|d_{t-1}\|^2$. Using the expression for $\eta_{t-1} $ together with Lemma~\ref{lem:technical-1} enables to show,
\begin{align*}
\rOne
\leq 
   4L^2\sum_{t=1}^{T}\frac{\|d_{t-1}\|^2}{(\sum_{i=1}^{t-1} \|d_i\|^2)^{2/3}} 
\leq 
   12L^2 (\sum_{t=1}^{T} \|d_t\|^2)^{1/3} ~.
\end{align*}
\textbf{Bounding $\rTwo$.} One can directly show that $\E[a_t^2\|g_t-\bg_t\|^2]\leq \E[a_t^2\|g_t\|^2]$. Using this together with the expression for $a_t$, it is possible to show that,
\begin{align*}
\E\rTwo
\leq 
  \E\sum_{t=1}^{T}\frac{\|g_t\|^2}{(1+\sum_{i=1}^{t-1} \|g_i\|^2)^{4/3}} 
\leq 
   C_1 ~.
\end{align*}
where $C_1$ is a constant, and the second inequality is due to a  lemma that we describe in the appendix.\\

\textbf{Bounding $\rThree$.} Since $\tau^*\in[T]$ is a bounded stopping time, and $M_t$ is a martingale difference sequence, then Doob's optional stopping theorem~\cite{levin2017markov} implies $\E\rThree = \E\sum_{t=1}^{\tau^*}M_t =0$. 

\textbf{Conclusion.} The above  together with Jensen's inequality for $U(z)=z^{1/3}$ defined over $\reals_{+}$, yields,
\begin{align}\label{eq:EpsBoundFirstPart}
\E\sum_{t=1}^{\tau^*}\|\eps_t\|^2 \leq 2C_1/\beta + 24(L^2/\beta) (\E\sum_{t=1}^{T} \|d_t\|^2)^{1/3}~.
\end{align}
\paragraph{Second Part: Bounding $\E\sum_{t=1}^{T}\|\eps_t\|^2$.}
Recall the error dynamics of \stormp appearing in Eq.~\eqref{eq:ErrorDynamics}. Dividing by $\sqrt{a_t}$, taking the square and summing up to some $T$ enables to bound,
\begin{align*}
\frac{1}{a_t}\|\eps_t\|^2 
&\leq
(\frac{1}{a_t}-1)\|\eps_{t-1}\|^2 + 2\frac{\|Z_t\|^2}{a_t} + 2a_t\|g_t-\bg_t\|^2 + Y_t
\end{align*}
where $Y_t = 2 \langle \frac{1 - a_t}{\sqrt a_t} \epsilon_{t-1}, \sqrt a_t (g_t - \bg_t) + \frac{1 - a_t}{\sqrt a_t} Z_t \rangle$ is a martingale difference sequence such $\E[Y_t\vert \H_{t-1}]=0$. 
Re-arranging the above and summing one can show,
$$
\sum_{t=1}^{T}\|\eps_{t-1}\|^2 \leq 
\underbrace{-\frac{1}{a_T}\|\eps_T\|^2}_{\rA} 
+
\underbrace{\sum_{t=1}^{T}(\frac{1}{a_{t+1}}-\frac{1}{a_t})\|\eps_t\|^2}_{\rB}  
+ 
2\underbrace{\sum_{t=1}^{T}\frac{\|Z_t\|^2}{a_t}}_{\rC}
 +
 2\underbrace{\sum_{t=1}^{T}a_t\|g_t-\bg_t\|^2}_{\rD}
  +
  \underbrace{ \sum_{t=1}^{T}Y_t}_{\rE}
$$
Now, due to the martingale property $\E\rE = 0$. Next, we focus on bounding term $\rB$,\\
\textbf{Bounding $\rB$.} 
Using the definition of $\tau^*$ one can show that $1/a_{t+1} \leq 1/\tbeta~;\forall t\leq \tau^*$, where $1/\tbeta: = (1/\beta^{3/2}+ G^2)^{2/3}$. Moreover, we can show,
$$
1/a_{t+1} - 1/a_t \leq 2/3~; \quad \forall t\geq \tau^*+1
$$
This enables to decompose and bound $\rB$ according to $\tau^*$,
\begin{align}
\sum_{t=1}^T(\frac{1}{a_{t+1}}-\frac{1}{a_t})\|\eps_t\|^2  &=
\sum_{t=1}^{\tau^*}(\frac{1}{a_{t+1}}-\frac{1}{a_t})\|\eps_t\|^2 +\sum_{t=\tau^*+1}^T(\frac{1}{a_{t+1}}-\frac{1}{a_t})\|\eps_t\|^2 \nonumber\\
&\leq
\frac{1}{\tbeta}\sum_{t=1}^{\tau^*}\|\eps_t\|^2 +\frac{2}{3}\sum_{t=\tau^*+1}^T\|\eps_t\|^2 
\leq
\frac{1}{\tbeta}\sum_{t=1}^{\tau^*}\|\eps_t\|^2 +\frac{2}{3}\sum_{t=1}^T\|\eps_t\|^2~.
\end{align}
This enables to use Eq.~\eqref{eq:EpsBoundFirstPart} to bound the expected value of term $\rB$.

From here the analysis of the other terms and bounding $\E\sum_{t=1}^{T}\|\eps_{t-1}\|^2$ is done similarly to our analysis of simplified \stormp.
\paragraph{Third Part: Bounding $\E\sum_{t=1}^{T}\|\bg_t\|^2$.} In this part we divide into two sub-cases depending whether
$\E\sum_{t=1}^T \norm{\eps_t}^2 \geq (1/2)\E\sum_{t=1}^T\|\bg_t\|^2 $ or not. And continue similarly to our analysis of simplified \stormp. The rest of the details appear in the appendix.
\end{proof}

\section{Conclusion}
We have presented a novel parameter-free and adaptive algorithm for non-convex optimization that obtains the optimal rate in the setting of expectation over smooth losses while adapting to variance in gradient estimates. Our approach suggests a new way to set the learning rate and momentum jointly and adaptively throughout the learning process, which might open up new avenues 
to both practical and theoretical developments in this direction.

\section*{Acknowledgments}
This work has received funding from the 
European Research Council (ERC) under the European Union's Horizon 2020 research and innovation programme (grant agreement no 725594 - time-data); 
Hasler Foundation Program: Cyber Human Systems (project number 16066); 
the Department of the Navy, Office of Naval Research (ONR) under a grant number N62909-17-1-2111; 
the Swiss National Science Foundation (SNSF) under  grant number 200021\_178865 / 1; and 
the Army Research Office under Grant Number W911NF-19-1-0404.
K.Y. Levy acknowledges support from the Israel Science Foundation (grant No. 447/20).
\newpage

\bibliography{refs}
\bibliographystyle{abbrvnat}


\iftrue
\newpage
\appendix

\section{Appendix}
\subsection{Proofs for Section~\ref{sec:StochasticAnalysisSimple}}
\subsubsection{Proof of Equation~\eqref{eq:SumOfGradsStochastic}}
We will use the following lemma that we prove in Section~\ref{sec:Proofem:SunGradSTORMP},
\begin{Lem} \label{lem:SunGradSTORMP}
For both \stormp and simplified \stormp the following holds,
\begin{align*}
	\sum_{t=1}^{T} \|\bg_t\|^2 
	 &\leq
	 \sum_{t=1}^{T} \|\eps_t\|^2 + 2Ba_{T+1}^{-1/3}(\sum_{t=1}^{T} \|d_t\|^2)^{1/3} + \frac{3}{2}L(\sum_{t=1}^T \|d_t\|^2)^{2/3} 
\end{align*}
Eq.~\eqref{eq:SumOfGradsStochastic} directly follows from this lemma by taking $a_{T+1} = 1/T^{2/3}$.
\end{Lem}

\subsubsection{Proof of Lemma~\ref{lem:technical-1}}\label{sec:Prooflem:technical-1}
We will prove the lemma by induction on $n$. The proof relies on the arguments in \citep{mcmahan2010adaptive} and generalizes it for any $p \in (0,1)$.
\begin{proof}
 For the base case of $n=1$, we can easily show that the hypothesis holds. 
	\begin{align*}
		\frac{b_1}{b_1^{p}} = b_1^{1-p} \leq \frac{1}{1-p} b_1^{1-p}
	\end{align*}
	Now, assuming that the hypothesis holds for some arbitrary number $n-1 > 1$, we want to show that it holds for $n$, too. Let us define $Z = \sum_{t=1}^{n} b_t$ and $x = b_n$. Then, using the inductive hypothesis for $n-1$,
	\begin{align*}
		\sum_{t=1}^{n} \frac{b_n}{ \left( \sum_{i=1}^{t} b_i \right)^{p} } &\leq \frac{1}{1-p} \left( \sum_{t=1}^{n-1} b_t \right)^{1-p} + \frac{b_n}{ \left( \sum_{t=1}^{n} b_t \right)^{p} } \\
		&= \frac{1}{1-p} ( Z - x )^{1-p} + \frac{x}{Z^{p}}
	\end{align*}
	Let us denote $h(x) = \frac{1}{1-p} ( Z - x )^{1-p} + \frac{x}{Z^{p}}$ is concave in $x$. What we need to show is that, for any choice of allowable $x$, $h(x) \leq \frac{1}{1-p} Z^{1-p} $. Specifically, we want to prove that
	\begin{align*}
		\max_{0 \leq x < Z} h(x) \leq \frac{1}{1-p} Z^{1-p}
	\end{align*}
	First, observe that $h(x)$ is a concave function, hence at the maximum the derivative evaluates to zero. Our aim is to find such $x$. Taking derivative wrt $x$,
	\begin{align*}
		\frac{d h(x)}{dx} = \frac{1}{Z^{p}} - \frac{1}{(Z - x)^{p}},
	\end{align*}
	which evaluates to zero when $x = 0$. Hence,
	\begin{align*}
		\max_{0 \leq x < Z} h(x) = h(0) = \frac{1}{1-p} Z^{1-p} = \frac{1}{1-p} \left( \sum_{t=1}^{n} b_t \right)^{1-p}
	\end{align*}
	which implies that the hypothesis is true:
	\begin{align*}
		\sum_{t=1}^{n} \frac{b_t}{ \left( \sum_{i=1}^{t} b_i \right)^{p} } \leq \frac{1}{1-p} \left( \sum_{t=1}^{n} b_t \right)^{1-p}.
	\end{align*}
\end{proof}

\subsubsection{Proof of Lemma~\ref{lem:SunGradSTORMP}}\label{sec:Proofem:SunGradSTORMP}
Using smoothness together with the update rule implies,
\begin{align*}
\Delta_{t+1}-\Delta_t =
	f(x_{t+1})-f(x_t) 
	&\leq
	 -\eta_t \bg_t^\top d_t + \frac{L\eta_t^2}{2} \|d_t\|^2 \\
	&=
	-\eta_t\|\bg_t\|^2 - \eta_t \bg_t^\top \eps_t +  \frac{L\eta_t^2}{2} \|d_t\|^2 \\
	&\leq
	-\eta_t\|\bg_t\|^2 + \frac{\eta_t}{2}\|\bg_t\|^2 +  \frac{\eta_t}{2}\|\eps_t\|^2+  \frac{L\eta_t^2}{2} \|d_t\|^2~,
\end{align*}
where we defined $\Delta_t: = f(x_t)-f(x^*)$. The second line above uses $d_t = \bg_t+\eps_t$, and the third line uses
$z^\top y \leq \frac{1}{2}(\|z\|^2+ \|y\|^2)$.

Re-arranging the above we get,
$$
\|\bg_t\|^2 \leq \|\eps_t\|^2 + \frac{2}{\eta_t}(\Delta_t - \Delta_{t+1}) +   L\eta_t \|d_t\|^2 \\
$$
Summing over $t$ gives,
\begin{align}\label{eq:SumOfGrads}
	\sum_{t=1}^{T} \|\bg_t\|^2 
	&\leq 
	\sum_{t=1}^{T} \|\eps_t\|^2
	-\frac{2}{\eta_T}\Delta_{T+1}
	  + 2\sum_{t=1}^{T} (\frac{1}{\eta_t} - \frac{1}{\eta_{t-1}})\Delta_t  +   L\sum_{t=1}^{T} \eta_t \|d_t\|^2 \nonumber\\
	  &\leq 
	\sum_{t=1}^{T} \|\eps_t\|^2
	  + 2B\sum_{t=1}^{T} (\frac{1}{\eta_t} - \frac{1}{\eta_{t-1}}) +   L\sum_{t=1}^{T} \frac{\|d_t\|^2}{(\sum_{i=1}^t \|d_t\|^2)^{1/3}} \nonumber\\
	   &\leq 
	  \sum_{t=1}^{T} \|\eps_t\|^2
	  + 2B\frac{1}{\eta_T} +   \frac{3}{2}L(\sum_{t=1}^T \|d_t\|^2)^{2/3} \nonumber\\
	  &\leq
	 \sum_{t=1}^{T} \|\eps_t\|^2 + 2B(\sum_{t=1}^{T} \|d_t\|^2/a_{t+1})^{1/3} + \frac{3}{2}L(\sum_{t=1}^T \|d_t\|^2)^{2/3} \nonumber\\
	 &\leq
	 \sum_{t=1}^{T} \|\eps_t\|^2 + 2B(1/a_{T+1})^{1/3}(\sum_{t=1}^{T} \|d_t\|^2)^{1/3} + \frac{3}{2}L(\sum_{t=1}^T \|d_t\|^2)^{2/3} 
\end{align}
The second line uses $\Delta_t\in[0,B]$, the third line uses Lemma~\ref{lem:technical-1}, and the last line uses the fact that $a_t$ is monotonically decreasing.

\newpage
\subsection{Full Analysis of \stormp (Algorithm~\ref{alg:stormp})}
The proof is composed of three parts: \textbf{(a)} In the first part we bound the cumulative expectation of errors $\E\sum_{t=1}^{\tau^*}\|\eps_t\|^2$, where $\eps_t: = d_t-\bg_t$, and $\tau^*$ is a stopping time after which we can ensure that $1/a_{t+1} -1/a_t\leq 2/3$. \textbf{(b)} In the second part we use our bound on $\E\sum_{t=1}^{\tau^*}\|\eps_t\|^2$ in order to bound the total sum of square errors, $\E\sum_{t=1}^{T}\|\eps_t\|^2$.
\textbf{(c)}  Then, in the last part we divide into two sub-cases the first where $\E\sum_{t=1}^T\|\eps_t\|^2\leq (1/2)\E\sum_{t=1}^T\|\bg_t\|^2$ and its complement. 
In one of these sub-cases we also use the smoothness of the objective together with the update rule, similarly to what we do in Eq.~\eqref{eq:OfflineCase}.

The different parts of the proof are divided between Section~\ref{sec:ProofPart1},~\ref{sec:ProofPart2}~, and ~\ref{sec:ProofPart3}.

\subsubsection{First Part: Bounding $\E\sum_{t=1}^{\tau^*}\|\eps_t\|^2$.}\label{sec:ProofPart1}
The  update rule for $d_t$ induces the following error dynamics,
\begin{align}\label{eq:ErrorDynamicsApp}
\eps_{t} = (1-a_t)\eps_{t-1} + a_t(g_t - \bg_t) + (1-a_t)Z_t
\end{align}
where $Z_t: = (g_t-\tg_{t-1}) - (\bg_t-\bg_{t-1})$. 

 Taking the square and summing up to some $\tau^*\in[T]$ enables to bound,
\begin{align*}
\sum_{t=1}^{\tau^*}\|\eps_t\|^2
& \leq \sum_{t=1}^{\tau^*}(1-a_t)^2\|\eps_{t-1}\|^2 +\sum_{t=1}^{\tau^*}\|(1-a_t)Z_t+ a_t(g_t - \bg_t) \|^2 
+ \sum_{t=1}^{\tau^*}M_t \\
& \leq \sum_{t=1}^{\tau^*}(1-a_t)\|\eps_{t-1}\|^2 +2\sum_{t=1}^{\tau^*}\|Z_t\|^2 + 2\sum_{t=1}^{\tau^*}a_t^2 \|g_t-\bg_t\|^2
+ \sum_{t=1}^{\tau^*}M_t~,
\end{align*}
where we used $\|b+c\|^2\leq 2\|b\|^2 +2\|b\|^2$, as well as $(1-a_t)\leq 1$.
We  have defined $\{M_t: =2(1-a_t)\eps_{t-1}^\top\left( (1-a_t)Z_t + a_t (g_t-\bg_t) \right) \}_{t\in[T]}$, and it is immediate to verify that $\{M_t\}_{t\in[T]}$ is a martingale difference sequence such $\E[M_t\vert \H_{t-1}]=0$, where $\H_t$ is the history to time $t$, i.e., $\H_t: = \{x_1,\xi_1,\xi_2,\xi_3\ldots, \xi_t\}$.

Now let us  define $\beta := \min\{1,1/G^4\}$, and  $\tau^* = \max\{t\in[T]: a_t\geq \beta\}$. Recalling that $a_{t+1}$ is measurable with respect to $\H_t$ implies that $\tau^*\in[T]$ is a stopping time.

Re-arranging the above and using the definition of $\tau^*$ implies,
\begin{align}\label{eq:EpsTauStarBasis}
\beta\sum_{t=1}^{\tau^*}\|\eps_t\|^2 \leq \|\eps_{\tau^*}\|^2 + \sum_{t=1}^{\tau^*-1}a_{t+1}\|\eps_t\|^2 \leq 2\underbrace{\sum_{t=1}^{T}\|Z_t\|^2}_{\rOne} + 2\underbrace{\sum_{t=1}^{T}a_t^2 \|g_t-\bg_t\|^2}_{\rTwo} + 
\underbrace{\sum_{t=1}^{\tau^*}M_t}_{\rThree}
\end{align}
where we used $\tau^*\leq T$, as well as $\beta\leq 1$.
Next we bound the expected value  of the above terms.\\
\textbf{Bounding $\rOne$.} Using smoothness property implies that $\|Z_t\| \leq 2L\|x_t-x_{t-1}\| = 2L\eta_{t-1}\|d_{t-1}\|$. Using the expression for $\eta_{t-1} $ together with Lemma~\ref{lem:technical-1} enables to show,
\begin{align*}
\rOne
\leq 
   4L^2\sum_{t=1}^{T}\frac{\|d_{t-1}\|^2}{(\sum_{i=1}^{t-1} \|d_i\|^2)^{2/3}} 
\leq 
   12L^2 (\sum_{t=1}^{T} \|d_t\|^2)^{1/3} ~.
\end{align*}
where the first inequality uses $\eta_t = 1/\left( \sum_{i=1}^t \|d_i\|^2/a_{i+1}\right)^{1/3} \leq 1/\left( \sum_{i=1}^t \|d_i\|^2\right)^{1/3}$.

\textbf{Bounding $\rTwo$.} 
Since $\E[g_t\vert H_{t-1}] = \bg_t$ and $a_t$ is measurable with respect to $\H_{t-1}$ it follows that
$$
\E[a_t^2\|g_t-\bg_t\|^2]\leq \E[a_t^2(\|g_t\|^2-\|\bg_t\|^2) ]\leq \E[a_t^2\|g_t\|^2 ]
$$
 Using this together with the expression for $a_t$, it is possible to show that,
\begin{align*}
\E\rTwo
\leq 
  \E\sum_{t=1}^{T}\frac{\|g_t\|^2}{(1+\sum_{i=1}^{t-1} \|g_i\|^2)^{4/3}} 
\leq 
   C_1 ~.
\end{align*}
where $C_1: = 12+2G^2$, and the last inequality is due to the following lemma  (recall $G$ is a bound on the gradient norms),
\begin{Lem} \label{lem:ConstSum_Delayed}
For any non-negative real numbers $a_1,\ldots, a_n\in[0,\amax]$,
\begin{align*}
\sum_{i=1}^n \frac{a_i}{(1+\sum_{j=1}^{i-1} a_j)^{4/3}} 
\le 
12+2\amax ~.
\end{align*}
\end{Lem}
We prove this lemma in Appendix~\ref{sec:Prooflem:ConstSum_Delayed}.

\textbf{Bounding $\rThree$.} Since $\tau^*\in[T]$ is a bounded stopping time, and $M_t$ is a martingale difference sequence, then Doob's optional stopping theorem~\cite{levin2017markov} implies $\E\rThree = \E\sum_{t=1}^{\tau^*}M_t =0$. 

\textbf{Conclusion.} Combining the above  bounds inside Eq.~\eqref{eq:EpsTauStarBasis} together with Jensen's inequality for $U(z)=z^{1/3}$ defined over $\reals_{+}$, yields,
\begin{align}\label{eq:EpsBoundFirstPartApp}
\E\sum_{t=1}^{\tau^*}\|\eps_t\|^2 \leq 2C_1/\beta + 24(L^2/\beta) (\E\sum_{t=1}^{T} \|d_t\|^2)^{1/3}~.
\end{align}

\subsubsection{Second Part: Bounding $\E\sum_{t=1}^{T}\|\eps_t\|^2$.}\label{sec:ProofPart2}
Recall the error dynamics of \stormp appearing in Eq.~\eqref{eq:ErrorDynamicsApp}. Dividing by $\sqrt{a_t}$, and taking the square gives,
\begin{align*}
\frac{1}{a_t}\|\eps_t\|^2 
&= (\frac{1}{a_t}-2 +a_t)\|\eps_{t-1}\|^2 + \|(1-a_t)\frac{Z_t}{\sqrt{a_t}} + \sqrt{a_t} (g_t-\bg_t)\|^2 + Y_t\\
&\leq
(\frac{1}{a_t}-1)\|\eps_{t-1}\|^2 + 2\frac{\|Z_t\|^2}{a_t} + 2a_t\|g_t-\bg_t\|^2 + Y_t~,
\end{align*}
where we used $a_t\in[0,1]$, and $(1-a_t)\in[0,1]$, as well as $\|b+c\|^2\leq 2\|b\|^2 + 2\|c\|^2$. We also defined $Y_t: = 2(\frac{1}{\sqrt{a_t}}-\sqrt{a_t})\eps_{t-1}^\top\left((1-a_t)\frac{Z_t}{\sqrt{a_t}} + \sqrt{a_t} (g_t-\bg_t)\right)$. Note that $\E[Y_t\vert \H_{t-1}]=0$;  therefore $Y_t$ is a martingale difference sequence.

Re-arranging the above and summing gives,
\begin{align}\label{eq:FullEpsBasis}
\sum_{t=1}^{T}\|\eps_{t-1}\|^2 \leq 
\underbrace{-\frac{1}{a_T}\|\eps_T\|^2}_{\rA} 
+
\underbrace{\sum_{t=1}^{T}(\frac{1}{a_{t+1}}-\frac{1}{a_t})\|\eps_t\|^2}_{\rB}  
+ 
2\underbrace{\sum_{t=1}^{T}\frac{\|Z_t\|^2}{a_t}}_{\rC}
 +
 2\underbrace{\sum_{t=1}^{T}a_t\|g_t-\bg_t\|^2}_{\rD}
  +
  \underbrace{ \sum_{t=1}^{T}Y_t}_{\rE}~.
\end{align}
Next, we bound the expected value if each of the above terms.

\textbf{Bounding $\rA$:} Since $a_T\leq 1$ we can bound $-{\E \norm{\eps_{T}}^2}/{a_T}\leq -\E \norm{\eps_{T}}^2$ \\

\textbf{Bounding $\rB$.} 
We will use the following lemma which we prove in Section~\ref{sec:Prooflem:AtsTauStar},
\begin{Lem}\label{lem:AtsTauStar}
The following holds,
$$
1/a_{t+1} \leq 1/\tbeta~;\forall t\leq \tau^*
$$
where $1/\tbeta := (1/\beta^{3/2}+ G^2)^{2/3}$~.

Moreover,
$$
1/a_{t+1} - 1/a_t \leq 2/3~; \quad \forall t\geq \tau^*+1
$$
\end{Lem}

Lemma~\ref{lem:AtsTauStar} enables to decompose and bound $\rB$ as follows,
\begin{align}
\sum_{t=1}^T(\frac{1}{a_{t+1}}-\frac{1}{a_t})\|\eps_t\|^2  &=
\sum_{t=1}^{\tau^*}(\frac{1}{a_{t+1}}-\frac{1}{a_t})\|\eps_t\|^2 +\sum_{t=\tau^*+1}^T(\frac{1}{a_{t+1}}-\frac{1}{a_t})\|\eps_t\|^2 \nonumber\\
&\leq
\frac{1}{\tbeta}\sum_{t=1}^{\tau^*}\|\eps_t\|^2 +\frac{2}{3}\sum_{t=\tau^*+1}^T\|\eps_t\|^2 
\leq
\frac{1}{\tbeta}\sum_{t=1}^{\tau^*}\|\eps_t\|^2 +\frac{2}{3}\sum_{t=1}^T\|\eps_t\|^2~.
\end{align}
Thus,
\begin{align}\label{eq:BtermBoundApp}
\E\rB 
&\leq
\E\frac{1}{\tbeta}\sum_{t=1}^{\tau^*}\|\eps_t\|^2 +\frac{2}{3}\E\sum_{t=1}^T\|\eps_t\|^2 \nonumber\\
&\leq
\frac{2C_1}{\beta\tbeta} + 24\frac{L^2}{\beta\tbeta} (\E\sum_{t=1}^{T} \|d_t\|^2)^{1/3} + \frac{2}{3}\E\sum_{t=1}^T\|\eps_t\|^2 
\end{align}
where we have used Eq.~\eqref{eq:EpsBoundFirstPartApp}.

\textbf{Bounding $\rC$.} Recalling that $\|Z_t\| \leq 2L\|x_t-x_{t-1}\| = 2L\eta_{t-1}\|d_{t-1}\|$, and using the expression for $\eta_{t-1} $ together with Lemma~\ref{lem:technical-1} enables to show,

Thus,
\begin{align}\label{eq:ZtBound1App}
\sum_{t=1}^{T}\frac{\|Z_t\|^2}{a_t} 
&\leq 
    4L^2\sum_{t=1}^T\eta_{t-1}^2 \|d_{t-1}\|^2/a_t  \nonumber\\
&=
4L^2\sum_{t=1}^T \frac{\|d_{t-1}\|^2/a_t}{\left(\sum_{i=1}^{t-1} \|d_i\|^2/a_{i+1} \right)^{2/3}} \nonumber\\
&\leq
12L^2\left(\sum_{t=1}^{T-1} \|d_t\|^2/a_{t+1} \right)^{1/3} \nonumber\\
&\leq
12L^2\frac{1}{a_{T}^{1/3}}\left(\sum_{t=1}^{T-1} \|d_t\|^2 \right)^{1/3} \nonumber\\
&\leq
12L^2\left(1+\sum_{t=1}^{T} \|g_t\|^2 \right)^{2/9}\left(\sum_{t=1}^{T} \|d_t\|^2 \right)^{1/3}~,
\end{align}
where we used the fact that $a_t$ is non-increasing.

Now, let us recall Young's inequality which states that for any $a,b>0$, and $p,q>1:~ \frac{1}{p}+\frac{1}{q}=1$ we have $ab\leq a^p/p + b^q/q$.
This implies that for any $a,b,\rho>0$ and $p=\frac{3}{2}, q=3$, we have,
\begin{align}\label{eq:YoungsIneqVersion}
a^{2/9}b^{1/3} = (a\rho^{9/2})^{2/9}(b/\rho^3)^{1/3} \leq  \frac{(a\rho^{9/2})^{2p/9}}{p} +  \frac{(b/\rho^3)^{q/3}}{q} = \frac{2}{3}a^{1/3}\rho^{3/2} + \frac{b}{3\rho^3}
\end{align}
Thus, taking $\rho=(512L^2)^{1/3}$, $a = 1+\sum_{t=1}^{T} \|g_t\|^2$, $b=\sum_{t=1}^{T} \|d_t\|^2 $, and using Young's inequality inside Eq.~\eqref{eq:ZtBound1App} implies,
\begin{align}\label{eq:ZtBound2App} 
\sum_{t=1}^{T}\frac{\|Z_t\|^2}{a_t} 
&\leq 
512L^3\left(1+\sum_{t=1}^{T} \|g_t\|^2\right)^{1/3}+
\frac{1}{128}\sum_{t=1}^{T} \|d_t\|^2
\end{align}
\paragraph{Bounding Term $\rD$:}
Note that  $a_t$ is measurable with respect to $\H_{t-1}$, and $\E[g_t\vert \H_{t-1}] = \bg_t$, therefore using smoothing gives,
$$
\E[a_t\|g_t-\bg_t\|^2] = \E[a_t(\|g_t\|^2 -\|\bg_t\|^2)] \leq  \E[a_t\|g_t\|^2 ] 
$$
 Thus,
\begin{align*}
\E[\rD]&:= \E\sum_{t=1}^{T}a_t \|g_t-\bg_t\|^2 \\
&\leq 
\E\sum_{t=1}^T a_t\|g_t\|^2\\
&=
\E\sum_{t=1}^T\frac{\|g_t\|^2}{(1+\sum_{i=1}^{t-1}\|g_i\|^2)^{2/3}} \\
&\leq 
G^2 + 6\E \left(1+\sum_{t=1}^{T} \|g_t\|^2\right)^{1/3},
\end{align*}
where the last line is due to the following lemma which is a modified and time-shifted version of Lemma~\ref{lem:technical-1}. We defer its proof to Appendix~\ref{sec:proof-shifted-technical-1}.
\begin{Lem} \label{lem:shifted-technical-1}
	Let $b_1, ..., b_n \in (0, b]$ be a sequence of non-negative real numbers for some positive real number $b$, $b_0 > 0$ and $p \in (0, 1)$ a rational number. Then,
	\begin{align*}
		\sum_{i=1}^{n} \frac{b_i}{ \left(b_0 + \sum_{j=1}^{i-1} b_j \right)^{p} } \leq \frac{b}{(b_0)^p} + \frac{2}{1-p} \left( b_0 + \sum_{i=1}^{n} b_i \right)^{1-p}
	\end{align*}
\end{Lem}

\paragraph{Bounding Term $\rE$:}
Since $\{Y_t\}_{t\in[T]}$ is a martingale difference sequence we have,
$$
\E\rE = \E\sum_{t=1}^TY_t = 0~.
$$

\paragraph{To Summarize:} Combining the above bounds inside Eq.~\eqref{eq:FullEpsBasis} we conclude that,
\begin{align}\label{eq:FinalFormSplitCases}
\frac{1}{3}\E\sum_{t=1}^{T}\|\eps_{t}\|^2 
 &\leq \frac{24L^2}{\beta\tbeta} \E(\sum_{t=1}^{T} \|d_t\|^2)^{1/3}  +\frac{2C_1}{\beta\tbeta} +2G^2  \nonumber\\
&+ 
(1024L^3+12)\E\left(1+\sum_{t=1}^{T} \|g_t\|^2\right)^{1/3}+
\frac{1}{64}\E\sum_{t=1}^{T} \|d_t\|^2  \nonumber\\
&\leq \frac{24L^2}{\beta\tbeta} (\E\sum_{t=1}^{T} \|d_t\|^2)^{1/3}  +\frac{2C_1}{\beta\tbeta} + 2G^2  \nonumber\\
&+ 
(1024L^3+12)\left(1+\E\sum_{t=1}^{T} \|g_t\|^2\right)^{1/3}+
\frac{1}{64}\E\sum_{t=1}^{T} \|d_t\|^2 
\end{align}
where we have used Jensen's inequality for the concave function $G(z) = z^{1/3}~; z\geq0$.

\subsubsection{Final Part of the Proof}\label{sec:ProofPart3}
We divide the final part of the proof into two subcases:

\textbf{Case 1: Assume $\E\sum_{t=1}^{T}\|\eps_{t}\|^2\geq (1/2) \E\sum_{t=1}^{T}\|\bg_{t}\|^2$.}
Using the condition of this subcase implies 
$$
\E\sum_t\|d_t\|^2 \leq 2\E\sum_{t=1}^{T} \|\bg_t\|^2 + 2\E\sum_{t=1}^{T} \|\eps_t\|^2 \leq 6\E\sum_{t=1}^{T} \|\eps_t\|^2
$$
Plugging this into Eq.~\eqref{eq:FinalFormSplitCases} gives,
\begin{align*}
\frac{1}{3}\E\sum_{t=1}^{T}\|\eps_{t}\|^2 
 &\leq \frac{24L^2}{\beta\tbeta} (6\E\sum_{t=1}^{T} \|\eps_t\|^2)^{1/3}  +\frac{2C_1}{\beta\tbeta} + 2G^2  \nonumber\\
&+ 
(1024L^3+12)\left(1+\sigma^2T+\E\sum_{t=1}^{T} \|\bg_t\|^2\right)^{1/3}+
\frac{6}{64}\E\sum_{t=1}^{T} \|\eps_t\|^2 
\end{align*}
where the first line uses $\E\|g_t\|^2 = \E\|\bg_t\|^2 + \E\|g_t-\bg_t\|^2 \leq \E\|\bg_t\|^2 +\sigma^2$.

Re-arranging and using $ \E\sum_{t=1}^{T}\|\bg_{t}\|^2\leq 2\E\sum_{t=1}^{T}\|\eps_{t}\|^2$ gives,
\begin{align*}
\frac{1}{5}\E\sum_{t=1}^{T}\|\eps_{t}\|^2 
 &\leq \frac{24L^2}{\beta\tbeta} (6\E\sum_{t=1}^{T} \|\eps_t\|^2)^{1/3}  +\frac{2C_1}{\beta\tbeta} + 2G^2  \nonumber\\
&+ 
(1024L^3+12)\left(1+\sigma^2T+2\E\sum_{t=1}^{T} \|\eps_t\|^2\right)^{1/3}
\end{align*}
And the above implies,
\begin{align}\label{eq:Case1AppFinished}
 \E\sum_{t=1}^{T}\|\bg_{t}\|^2\leq 2\E\sum_{t=1}^{T}\|\eps_{t}\|^2 \leq O\left(1+\frac{C_1}{\beta\tbeta} +\left(\frac{L^2}{\beta\tbeta}\right)^{3/2}+ G^2 + L^3 + L^{9/2}+L^3\sigma^{2/3}T^{1/3} \right)
\end{align}

\textbf{Case 2: Assume $\E\sum_{t=1}^{T}\|\eps_{t}\|^2\leq (1/2) \E\sum_{t=1}^{T}\|\bg_{t}\|^2$.}
Using Lemma~\ref{lem:SunGradSTORMP} we get,
\begin{align}\label{eq:Case2App1}
	\sum_{t=1}^{T} \|\bg_t\|^2 
	 &\leq
	 \sum_{t=1}^{T} \|\eps_t\|^2 + 2B(1+\sum_{t=1}^{T} \|g_t\|^2)^{2/9}(\sum_{t=1}^{T} \|d_t\|^2)^{1/3} + \frac{3}{2}L(\sum_{t=1}^T \|d_t\|^2)^{2/3}  \nonumber\\
	 &\leq
	  \sum_{t=1}^T \|\eps_t\|^2+ \frac{3}{2}L(\sum_{t=1}^T \|d_t\|^2)^{2/3}+20B^{3/2}(1+\sum_{t=1}^{T}\|g_t\|^2)^{1/3}+\frac{1}{64}\sum_{t=1}^T \|d_t\|^2
\end{align}
where the second line uses a version of Young's inequality appearing in Eq.~\eqref{eq:YoungsIneqVersion} together with
 taking $\rho:=(128B/3)^{1/3}$, $a: = 1+\sum_{t=1}^{T} \|g_t\|^2$, and $b:=\sum_{t=1}^{T} \|d_t\|^2 $.

Using the condition of this subcase implies 
$$
\E\sum_t\|d_t\|^2 \leq 2\E\sum_{t=1}^{T} \|\bg_t\|^2 + 2\E\sum_{t=1}^{T} \|\eps_t\|^2 \leq 3\E\sum_{t=1}^{T} \|\bg_t\|^2
$$

Taking expectation of Eq.~\eqref{eq:Case2App1}, and using the above together with the condition gives,
\begin{align*}\label{eq:Case2App1}
	\E\sum_{t=1}^{T} \|\bg_t\|^2 
	 &\leq
	  \E\sum_{t=1}^T \|\eps_t\|^2+ \frac{3}{2}L(\E\sum_{t=1}^T \|d_t\|^2)^{2/3}+20B^{3/2}(1+\E\sum_{t=1}^{T}\|g_t\|^2)^{1/3}+\frac{1}{64}\E\sum_{t=1}^T \|d_t\|^2 \\
	  &\leq
	  \left(\frac{1}{2} + \frac{3}{64}\right)\E\sum_{t=1}^{T} \|\bg_t\|^2 + \frac{3}{2}L(3\E\sum_{t=1}^T \|\bg_t\|^2)^{2/3}+20B^{3/2}(1+\sigma
	  ^2T+\E\sum_{t=1}^{T}\|\bg_t\|^2)^{1/3}
\end{align*}
where we have used Jensen's inequality for the functions $z^{1/3},z^{2/3}$ defined over $\reals_{+}$, We also uses 
$\E\|g_t\|^2 = \E\|\bg_t\|^2 + \E\|g_t-\bg_t\|^2 \leq \E\|\bg_t\|^2 +\sigma^2$.

Re-arranging the above we conclude that,
$$
\E\sum_{t=1}^{T} \|\bg_t\|^2 \leq 6L(3\E\sum_{t=1}^T \|\bg_t\|^2)^{2/3}+ 80B^{3/2}(1+\sigma
	  ^2T+\E\sum_{t=1}^{T}\|\bg_t\|^2)^{1/3}
$$
This immediately implies that,
\begin{equation}\label{eq:Case2AppFinished}
\E\sum_{t=1}^{T} \|\bg_t\|^2 \leq O(1+L^3 + B^{9/4}+B^{3/2}\sigma^{2/3} T^{1/3})
\end{equation}

\textbf{Concluding}
From Equations~\eqref{eq:Case1AppFinished},~\eqref{eq:Case2AppFinished} it follows that,
\begin{equation}\label{eq:Case2AppFinished}
\E\sum_{t=1}^{T} \|\bg_t\|^2 \leq O(M^2+\kappa^2\sigma^{2/3} T^{1/3})
\end{equation}
where $\kappa^2: = B^{3/2}+L^3$, and $M^2:= 1+L^{9/2}+B^{9/4}+G^{10} + (L^2G^8)^{3/2}$.

Using the definition if $\bx_T$ together with Jensen's inequality gives,
$$
\E\|\nabla f(\bx_T)\| = O\left( \frac{M}{\sqrt{T}} + \frac{\kappa \sigma^{1/3}}{T^{1/3}}\right)~.
$$
which concludes the proof.

\newpage
\subsection{Additional Proofs}
\subsubsection{Proof of Lemma~\ref{lem:ConstSum_Delayed}}
\label{sec:Prooflem:ConstSum_Delayed}
\begin{proof}[Proof of Lemma~\ref{lem:ConstSum_Delayed}]
Lets define,
$$
N_0 = \min \left\{i\in[n]: \sum_{j=1}^{i-1}a_j \geq \amax  \right\}~.
$$
Thus, we can decompose the sum as follows,
\begin{align*}
\sum_{i=1}^n \frac{a_i}{(1+\sum_{j=1}^{i-1} a_j)^{4/3}} 
&=\sum_{i=1}^{N_0-1} \frac{a_i}{(1+\sum_{j=1}^{i-1} a_j)^{4/3}}  + \sum_{i=N_0}^{n} \frac{a_i}{(1+\sum_{j=1}^{i-1} a_j)^{4/3}}\\
&\leq
\sum_{i=1}^{N_0-1} a_i +  \sum_{i=N_0}^{n} \frac{a_i}{(1+\sum_{j=1}^{N_0-1} a_j  + \sum_{j=N_0}^{i-1} a_i )^{4/3}}\\
&\leq
2\amax + \sum_{i=N_0}^{n} \frac{a_i}{(1+\amax + \sum_{j=N_0}^{i-1} a_i )^{4/3}}\\
&\leq
2\amax +  \sum_{i=N_0}^{n}\frac{a_i}{(1+a_i + \sum_{j=N_0}^{i-1} a_i )^{4/3}}\\
&\leq
2\amax+12
\end{align*}
where the second and third lines use the definition of $N_0$ and definition of $\amax$, the fourth line uses $a_i\leq \amax$, and the last line uses the following helper lemma that we prove in Section~\ref{sec:Prooflem:ConstSum}.
\begin{Lem}\label{lem:ConstSum}
For any non-negative real numbers $a_1,\ldots, a_n\in[0,\amax]$,
\begin{align*}
\sum_{i=1}^n \frac{a_i}{(1+\sum_{j=1}^{i} a_j)^{4/3}} 
\leq 
12 ~.
\end{align*}
\end{Lem}
\end{proof}

\subsubsection{Proof  of Lemma~\ref{lem:ConstSum}} \label{sec:Prooflem:ConstSum}
\begin{proof}[Proof of Lemma~\ref{lem:ConstSum}]
Define,
$$
N_0 = \max \left\{i\in[n]: \sum_{j=1}^{i}a_j \leq 2  \right\}~.
$$
as well as for any $k\geq 1$
$$
N_k = \max \left\{i\in[n]: 2^{k}<\sum_{j=1}^{i}a_j \leq  2^{k+1} \right\}~.
$$
Now lets split the sum according to the $N_k$'s
\begin{align*}
\sum_{i=1}^n \frac{a_i}{(1+\sum_{j=1}^{i} a_j)^{4/3}} 
&=
\sum_{i=1}^{N_0} \frac{a_i}{(1+\sum_{j=1}^{i} a_j)^{4/3}}  +\sum_{k=1}^\infty\sum_{i = N_{k-1}+1}^{N_k}  \frac{a_i}{(1+\sum_{j=1}^{i} a_j)^{4/3}} \\
&\leq
\sum_{i=1}^{N_0} a_i + \sum_{k=1}^\infty \frac{1}{(2^{k})^{4/3}}\sum_{i = 1}^{N_k} a_i \\
&\leq
2 + \sum_{k=1}^\infty \frac{2^{k+1}}{(2^{k})^{4/3}} \\
&=
2 + \sum_{k=1}^\infty \frac{2^{k+1}}{(2^{k})^{4/3}} \\
&=
2 + 2\sum_{k=1}^\infty \left(\frac{1}{2^{1/3}}\right)^k \\
&\leq
2+2\cdot\frac{1}{1-2^{-1/3}} \\
&\leq 12~.
\end{align*}
\end{proof}

\subsubsection{Proof of Lemma~\ref{lem:AtsTauStar}}\label{sec:Prooflem:AtsTauStar}
\begin{proof}
The lemma has two parts.
\paragraph{Proof of first part.}
Recalling that   $\tau^* = \max\{t\in[T]: a_t\geq \beta\}$ for $\beta = \min\{1,1/G^4\}$ implies that
${1}/{a_t} \leq 1/\beta~;\forall t\leq \tau^*$. Moreover, using the definition of $a_t$ and boundedness of gradients we obtain,
$$
\left({1}/{a_{\tau^*+1}}\right)^{3/2} = \left({1}/{a_{\tau^*}}\right)^{3/2} + \|g_{\tau^*}\|^2 \leq \frac{1}{\beta^{3/2}} + G^2
$$
Defining  $\frac{1}{\tbeta} : = \left(\frac{1}{\beta^{3/2}} +G^2 \right)^{2/3}$ implies that,
$$
{1}/{a_t} \leq 1/\tbeta~;\qquad\forall t\leq \tau^*+1~.
$$

\paragraph{Proof of second part.}
First note that the function $H(y): = y^{2/3}$ is concave over $\reals_{+}$. Applying the gradient inequality for concave functions imply that,
$$
\forall y_1,y_2\geq 0~; H(y_2) - H(y_1)\leq \nabla H(y_1)^\top (y_2-y_1) = \frac{2}{3}\frac{1}{y_1^{1/3}} \cdot(y_2-y_1) ~.
$$
Therefore, for any $t\geq \tau^*+1$
\begin{align*}
\frac{1}{a_{t+1}} - \frac{1}{a_t} &= (1+\sum_{i=1}^{t-1}\|g_i\|^2 + \|g_t\|^2)^{2/3} - (1+\sum_{i=1}^{t-1}\|g_i\|^2)^{2/3}\\
& \leq
 \frac{2}{3}\frac{\|g_t\|^2}{(1+\sum_{i=1}^{t-1}\|g_i\|^2)^{1/3}} \\
 &=
 \frac{2}{3}\sqrt{a_t}\|g_t\|^2 \\
 &\leq 
  \frac{2}{3}\sqrt{\beta}G^2 \\
  &\leq 
  \frac{2}{3}~.
\end{align*}
where the fourth line uses the definition of $\tau^*$, and the last line uses the definition of $\beta$.

\end{proof}

\subsubsection{Proof of Lemma~\ref{lem:shifted-technical-1}}
\label{sec:proof-shifted-technical-1}
\begin{proof}
	Let $b_1, ..., b_n \in (0, b]$ be a sequence of non-negative real numbers for some positive real number $b$, $b_0 > 0$ and $p \in (0, 1)$ a rational number. Then,
	\begin{align*}
		\sum_{i=1}^{n} \frac{b_i}{ \left(b_0 + \sum_{j=1}^{i-1} b_j \right)^{p} } \leq \frac{b}{(b_0)^p} + \frac{2}{1-p} \left(b_0 +  \sum_{i=1}^{n} b_i \right)^{1-p}
	\end{align*}
	
	The proof of this lemma relies on the arguments of Lemma A.1 from \citep{bach2019universal} and makes use of Lemma~\ref{lem:technical-1} we proved earlier. We consider two cases for the proof depending on whether $b_0 \leq b$ or $b_0 \geq b$.
	
	\textbf{Case 1 : ${b_0 \geq b}$}.
	\begin{align*}
		\sum_{i=1}^{n} \frac{b_i}{ \left(b_0 + \sum_{j=1}^{i-1} b_j \right)^{p} } &\leq \sum_{i=1}^{n} \frac{b_i}{ \left(b + \sum_{j=1}^{i-1} b_j \right)^{p} } \\
		&\leq \sum_{i=1}^{n} \frac{b_i}{ \left(\sum_{j=1}^{i} b_j \right)^{p} } \\
		&\leq \frac{1}{1-p} \left( \sum_{i=1}^{n} b_i \right)^{1-p}\\
		&\leq \frac{b}{(b_0)^p} + \frac{2}{1-p} \left( b_0 + \sum_{i=1}^{n} b_i \right)^{1-p}\\
	\end{align*}
	
	\textbf{Case 2 : ${b_0 \leq b}$}.\\
	Let us denote a time variable
	$$
	T_0 = \min \left\{ i \in [n] : \sum_{j=1}^{i-1} b_j \geq b \right\}
	$$
	Then, we could separate the summation as
	\begin{align*}
		\sum_{i=1}^{n} \frac{b_n}{ (b_0 + \sum_{j=1}^{i-1} b_j )^{p} } &= \sum_{i=1}^{T_0 - 1} \frac{b_n}{ (b_0 + \sum_{j=1}^{i-1} b_j )^{p} } + \sum_{i=T_0}^{n} \frac{b_n}{ (b_0 + \sum_{j=1}^{i-1} b_j )^{p} } \\
		&\leq \frac{1}{(b_0)^p} \sum_{i=1}^{T_0 - 1} b_n + \sum_{i=T_0}^{n} \frac{b_n}{ (\frac{1}{2} \sum_{j=1}^{i-1} b_j + \frac{1}{2} \sum_{j=1}^{i-1} b_j )^{p} } \\
		&\leq \frac{b}{(b_0)^p} + \sum_{i=T_0}^{n} \frac{b_n}{ (\frac{1}{2} b + \frac{1}{2} \sum_{j=1}^{i-1} b_j )^{p} } \tag{Use definition of $T_0$}\\
		&\leq \frac{b}{(b_0)^p} + 2 \sum_{i=T_0}^{n} \frac{b_n}{ (\sum_{j=1}^{i} b_j )^{p} } \tag{Use $b_i \leq b$} \\
		&\leq \frac{b}{(b_0)^p} + \frac{2}{1-p} \left( \sum_{i=T_0}^{n} b_i \right)^{1-p} \tag{Use Lemma~\ref{lem:technical-1}}\\
		&\leq \frac{b}{(b_0)^p} + \frac{2}{1-p} \left( b_0 + \sum_{i=1}^{n} b_i \right)^{1-p}
	\end{align*}

\end{proof}

\newpage
\subsection{Numerical Results}
In this section we provide numerical performance of \stormp for a multi-class classification task.
Specifically, we train ResNet34 architecture on CIFAR10 dataset using SGD with momentum, \storm and \stormp, as well as AdaGrad and Adam.
We implemented the whole setup in \emph{pytorch}~\cite{pytorch2019} retrieving the model and the dataset from \emph{torchvision} package.
We executed the experiments on NVIDIA DGX infrastructure. Specifically, our code ran on NVIDIA A100-SXM4-40GB graphics card.
We use mini-batches of 100 samples both for training and testing, whiling using the default train/test data split provided in the package.

To be fair to all methods, we fixed all the parameters to their default value except for the learning rate.
Then, we executed an initial learning rate sweep over the same logarithmic range for all the algorithms.
All methods use a constant learning rate schedule without any heuristic strategies.
All methods are run with the best performing initial learning rate after tuning and the results for a single run are presented in Figure~\ref{fig:resnet}.
In the plots, epoch refers to the number of passes over dataset, \emph{not} number of gradient calls.
Per iteration cost of \storm and \stormp are twice that of other methods with respect to forward/backward passes.

\begin{figure}[H]
\centering
\begin{subfigure}[t]{.48\textwidth}
  \centering
  \includegraphics[width=.95\linewidth]{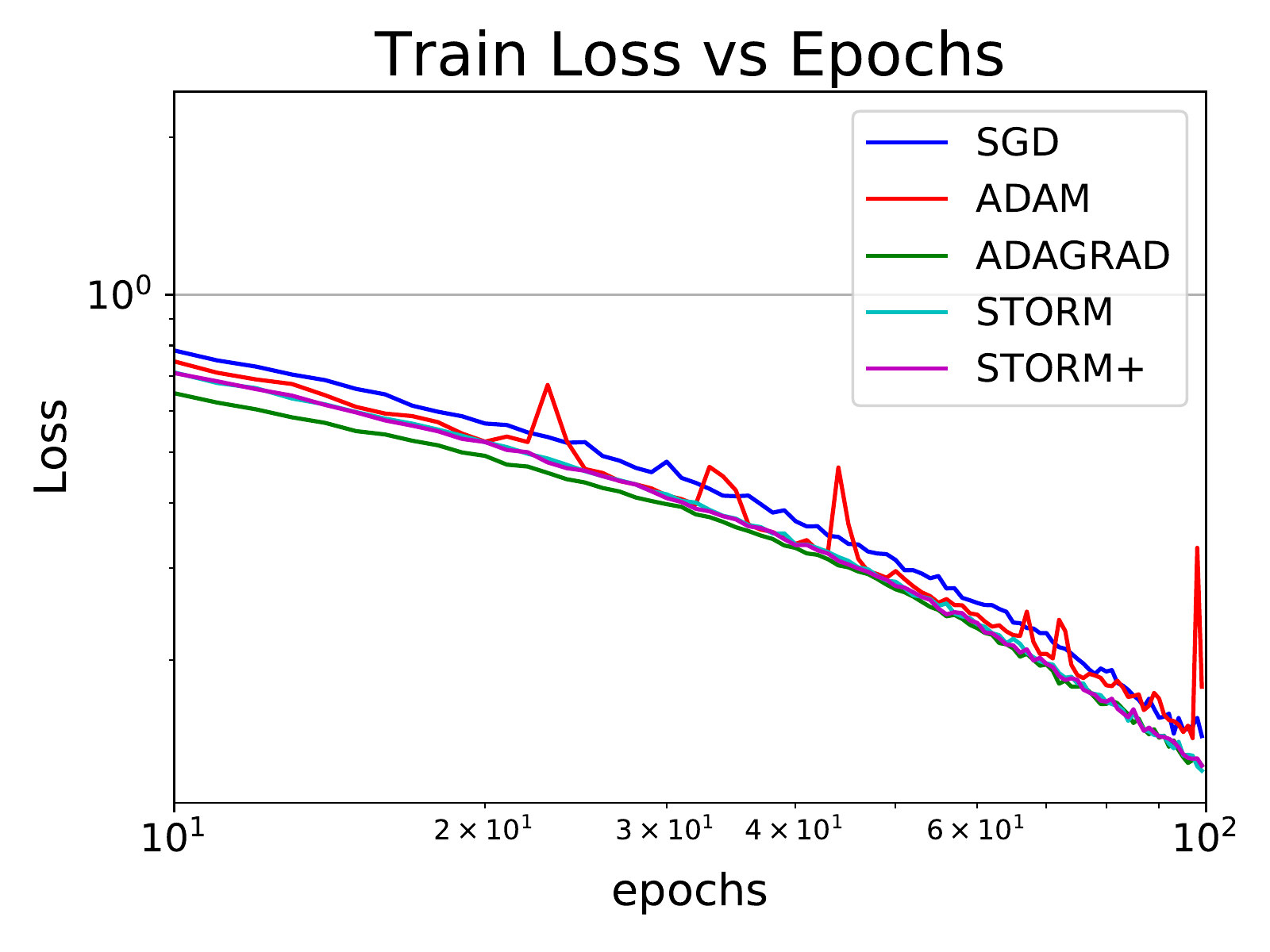}  
  \caption{Training loss in log-log scale}
  \label{fig:sub-train-loss}
\end{subfigure}%
\begin{subfigure}[t]{.48\textwidth}
  \centering
  \includegraphics[width=.95\linewidth]{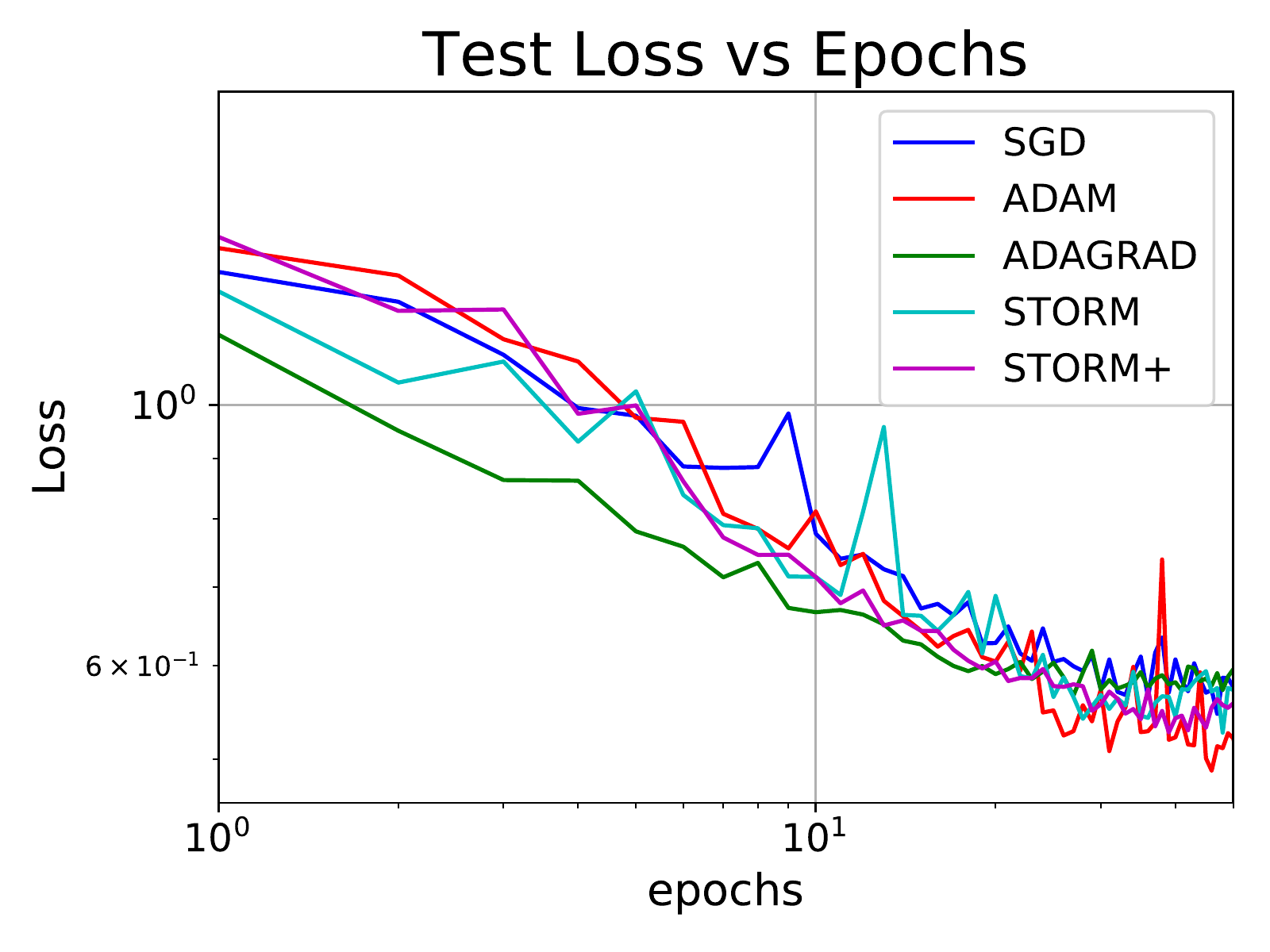}  
  \caption{Test loss in log-log scale}
  \label{fig:sub-test-loss}
\end{subfigure} \hspace{1mm}\vspace{6mm}

\begin{subfigure}[t]{.48\textwidth}
  \centering
  \includegraphics[width=.95\linewidth]{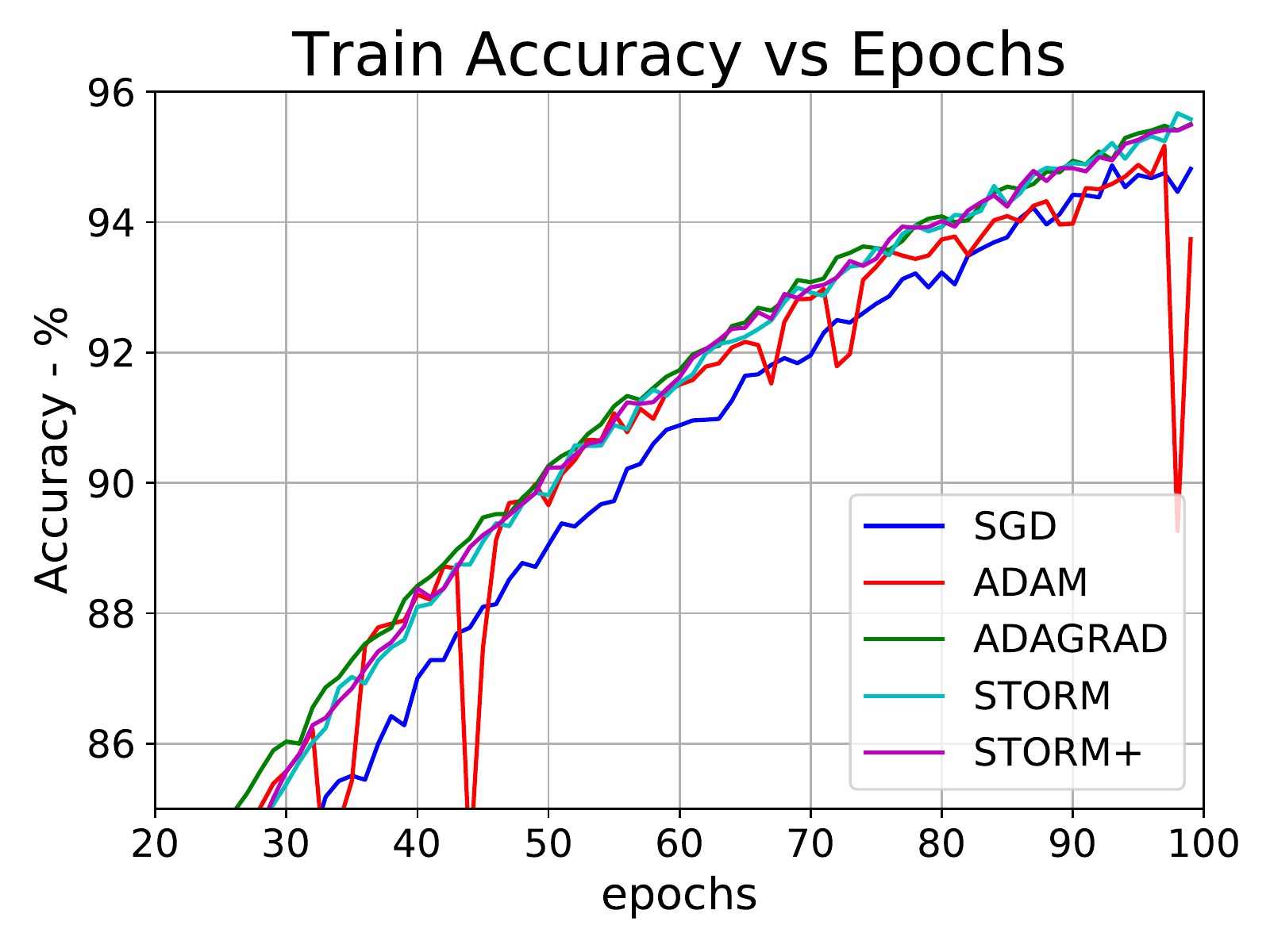}  
  \caption{Training accuracy}
  \label{fig:sub-train-acc}
\end{subfigure}%
\begin{subfigure}[t]{.48\textwidth}
  \centering
  \includegraphics[width=.95\linewidth]{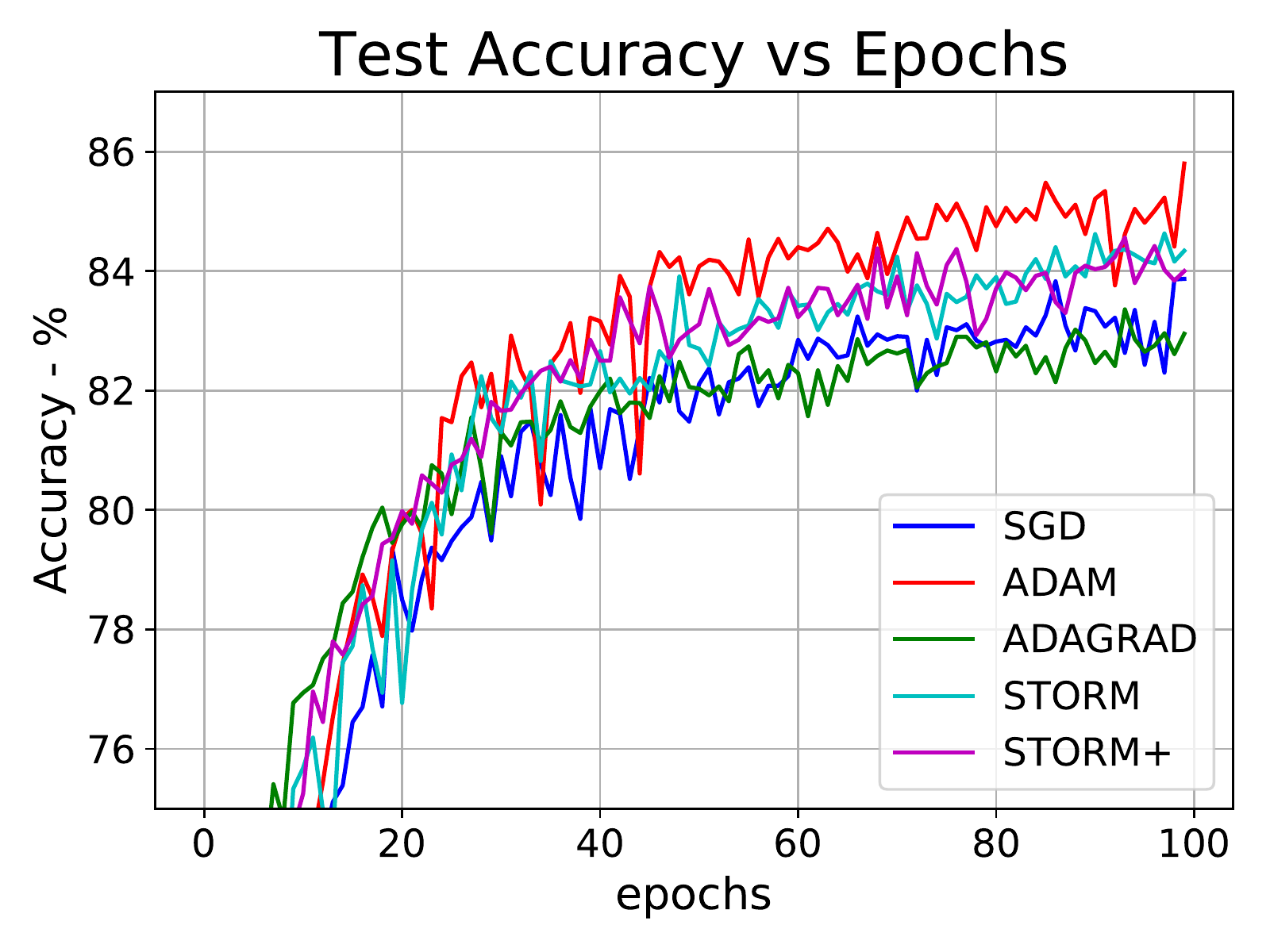}  
  \caption{Test accuracy}
  \label{fig:sub-test-acc}
\end{subfigure}
\caption{Comparison of SGD and adaptive methods, Resnet34 on CIFAR10}
\label{fig:resnet}
\end{figure}

The results do not exhibit a noticeable practical advantage for \stormp, however, they verify that it achieves comparable performance with respect to other adaptive methods. The performance of \storm and \stormp are quite close to each other under all 4 metrics. In the training phase, \storm and \stormp seem to outperform other methods by a small margin, both in training accuracy and training loss.
Adam and SGD seem to achieve a relatively small training accuracy and relatively large training loss compared to other methods.
In the test phase, we observe a different picture where Adam generalizes slightly better than other methods, followed by \storm and \stormp as we could see in Figure~\ref{fig:sub-test-acc}.

In terms of ease of tuning, provably, \stormp does not require the knowledge of \emph{any} problem parameters to operate and only initial step-size tuning suffices, while \storm additionally needs to tune the initial momentum parameter as, in theory, it requires the knowledge of smoothness and bound on the gradients.
Adam would need tuning for its moving average parameters $\beta_1$ and $\beta_2$, while SGD has a momentum parameter which is subject to a search over admissible values.
Similar to \stormp, AdaGrad does not require tuning beyond initial learning rate. 

\fi

\end{document}

%% file: defs.tex
\newtheorem{Thm}{Theorem}
\newtheorem{Lem}[Thm]{Lemma}

\DeclareMathOperator*{\argmin}{arg\,min}

\def\reals{{\mathcal R}}

\newcommand{\D}{{\mathcal{D}}}

\newcommand{\ignore}[1]{}

\def\reals{{\mathbb R}}

\def\bold0{\mathbf{0}}

\newcommand\E{\mbox{\bf E}}

\def\bx{\mathbf{x}}

\newcommand{\eps}{\epsilon}

\newcommand{\norm}[1]{\left\| {#1} \right\|}